\documentclass{birkjour_t2}
\usepackage{subfigure}
\usepackage{multirow}

\usepackage{color}
\usepackage{hyperref}
\hypersetup{colorlinks=true,linkcolor=blue,citecolor=red} %

\newtheorem{prop}{Proposition}[section]
\newtheorem{remark}{Remark}[section]
\newtheorem{lemma}{Lemma}[section]
\newtheorem{theorem}{Theorem}[section]
\newtheorem{corollary}{Corollary}[section]
\def\vec#1{{\bf #1}}
\def\bm#1{{\bf #1}}
\def\pt{\partial}

\usepackage{lipsum}
\usepackage{bm}

\newcommand\blfootnote[1]{%
\begingroup
\renewcommand\thefootnote{}\footnote{#1}%
\addtocounter{footnote}{-1}%
\endgroup
}

\begin{document}

\markboth{}{}

%

%

\title[On physical-constraints-preserving
	schemes for special relativistic magnetohydrodynamics]
{On physical-constraints-preserving
	schemes for special\\ relativistic magnetohydrodynamics with a general equation of state}

\author[Kailiang Wu]{Kailiang Wu}


\address{Department of Mathematics,\\
	The Ohio State University, Columbus, OH 43210, USA\\
wu.3423@osu.edu}

\author[Huazhong Tang]{Huazhong Tang$^{\dagger}$
\blfootnote{$^{\dagger}$Corresponding author. Tel:~+86-10-62757018;
Fax:~+86-10-62751801.}}

\address{HEDPS, CAPT and LMAM, School of Mathematical Sciences,\\
	Peking
	University, Beijing 100871, P.R. China\\
School of Mathematics and Computational Science,\\
Xiangtan University, Xiangtan 411105,\\
Hunan Province, P.R. China\\
hztang@math.pku.edu.cn}

\maketitle


\begin{abstract}
The paper studies the physical-constraints-preserving  (PCP) schemes
for multidimensional special relativistic magnetohydrodynamics 
with a general equation of state (EOS) on more general
meshes.
It is an extension of the work (Ref.\ \cite{WuTangM3AS2017})
which focuses on the ideal EOS and uniform Cartesian meshes.
The general EOS without a special expression poses some additional difficulties
in
%
discussing the mathematical  properties of  admissible state set with
the physical constraints on
the fluid velocity, density and pressure.
Rigorous analyses are provided for the PCP property of finite volume or discontinuous Galerkin schemes with  the  Lax-Friedrichs (LxF) type flux on a general mesh with non-self-intersecting polytopes. Those are built on a more general form of generalized
LxF splitting property and a different convex decomposition technique.
It is shown in theory that the PCP property is closely connected with a discrete divergence-free condition, which is proposed on the general mesh and milder than that in Ref.\ \cite{WuTangM3AS2017}.

\noindent


\end{abstract}

\section{Introduction}

The paper is concerned with physical-constraints-preserving (PCP) numerical methods for the special relativistic magnetohydrodynamics (RMHD). The  governing equations of $d$-dimensional special RMHDs in the   laboratory frame   can be written in the divergence form
\begin{equation}\label{eq:RMHD1D}
\frac{{\partial \vec U}}{{\partial t}} +
\sum_{i=1}^d {\frac{{\partial {\vec F_i}(\vec U)}}{{\partial x_i}}}  = {\bf 0},
\end{equation}
together with the divergence-free condition on the magnetic field
$\vec B=(B_1,B_2,B_3)$, i.e.,
\begin{equation}\label{eq:2D:BxBy0}
\sum_{i=1}^d\frac{\partial B_i } {\partial x_i}=0,
\end{equation}
where the conservative vector $\vec U = \big( D,\vec m,\vec B,E \big)^{\top}$  and the flux in the $x_i$-direction $\vec F_i(\vec U)$ is defined by
\begin{equation*}
\vec F_i(\vec U) = \bigg( D v_i,  v_i \vec m  -  B_i \Big( W^{-2} \vec B + (\vec v \cdot \vec B) \vec v \Big)  + p_{tot}  \vec e_i, v_i \vec B - B_i \vec v  ,m_i \bigg)^{\top},~i=1,\cdots,d,
\end{equation*}
with the mass density $D = \rho W$, the momentum density (row) vector $\vec m = \rho h{W^2}\vec v + |\vec B{|^2} \vec v - (\vec v \cdot \vec B)\vec B$, the energy density $E=\rho h W^2 - p_{tot} +|\vec B|^2$, and
the row vector $\vec e_i$ denoting the $i$-th row of the unit matrix of size $3$.
Here $\rho$ is the rest-mass density, $\vec v=(v_1,v_2,v_3)$ denotes
the fluid velocity vector,  $p_{tot}$ is the total pressure containing the gas pressure $p$ and  magnetic pressure $p_m:=\frac12 \left(W^{-2} |\vec B|^2 +(\vec v \cdot \vec B)^2 \right)$,
$W=1/\sqrt{1- v^2}$ is the Lorentz factor with $v:=\left(v_1^2+v_2^2+v_3^2\right)^{1/2}$,
$h$ is the specific enthalpy defined by
$$h = 1 + e + \frac{p}{\rho},$$
with units in which the speed of light $c$ is equal to one, and $e$ is the specific internal energy.
It can be seen that the  variables
$\vec m$ and $E$ depend on the magnetic field $\vec B$  nonlinearly.
Note that similar to the study in the literature, the $d$-dimensional discussion in this paper
focuses on the general setup with ${\bf v}, {\bf B} \in {\mathbb R}^3$ instead of ${\mathbb R}^d$.

The system \eqref{eq:RMHD1D} {is closed} mathematically,
only if an additional thermodynamic equation relating state variables,
i.e., the so-called equation of
state (EOS), is given. A general EOS may be expressed as
\begin{equation}
\label{eq:EOS:h}
h = h(p,\rho)= 1+e(p,\rho) + p/\rho,
\end{equation}
which will be discussed in detail in the next section.
A simple example is  the ideal EOS
\begin{equation}\label{eq:iEOS}
h = 1 + \frac{\Gamma p }{(\Gamma -1)\rho},
\end{equation}
where the adiabatic index $\Gamma \in (1,2]$.
The system \eqref{eq:RMHD1D} takes into account the relativistic description for the dynamics of electrically-conducting
fluid (plasma) at nearly speed of light in vacuum in the presence of magnetic fields. The relativistic magneto-fluid flow appears in investigating numerous astrophysical phenomena from stellar to galactic scales, e.g., core collapse super-novae, coalescing neutron stars, X-ray binaries, active galactic nuclei, formation of black holes, super-luminal jets and gamma-ray bursts etc.
However, due to the relativistic effect, especially the appearance of Lorentz factor, the system \eqref{eq:RMHD1D}   involves strong
nonlinearity, making its analytic treatment extremely difficult. Numerical simulation is a primary and powerful approach to improve our understanding of the physical
mechanisms in the RMHDs. In comparison with the non-relativistic MHD case, the numerical difficulties mainly come from highly nonlinear coupling between the RMHD equations in \eqref{eq:RMHD1D}, which leads to no explicit expression of the primitive variables $(\rho,\vec v,p) $ and the flux $\vec F_i$ in terms of $\vec U$.

Since nearly 2000s, numerical study of the RMHDs has attracted
considerable attention, and various modern shock-capturing methods have been developed for the RMHD equations. They include but are not limited to: the Godunov-type scheme based on  the linear Riemann solver \cite{GodunovRMHD}, the total variation diminishing scheme \cite{Balsara2001},
the third-order accurate central-type scheme based on two-speed approximate Riemann solver \cite{Zanna:2003}, the high-order kinetic flux-splitting method \cite{Qamar2005},
the exact Riemann solver \cite{Giacomazzo2006},
the HLLC (Harten-Lax-van Leer-contact) type schemes \cite{Honkkila:2007,Kim2014,MignoneHLLCRMHD},
the adaptive methods with mesh refinement \cite{Anderson:2006,Host:2008},
the adaptive moving mesh method \cite{HeTang2012RMHD},
the locally divergence-free Runge-Kutta discontinuous Galerkin (RKDG) method and
exactly divergence-free central RKDG method with the weighted essentially non-oscillatory ({WENO}) limiters \cite{ZhaoTang2016},
the ADER (Arbitrary high order schemes using DERivatives) DG method \cite{Zanotti2015},  and the ADER-WENO type schemes with subluminal reconstruction \cite{BalsaraKim2016}, etc. The readers are also referred to the
early review articles \cite{font2008,Marti2015}. Besides the
standard difficulty in solving the nonlinear hyperbolic systems,
an additional numerical challenge for the RMHD system \eqref{eq:RMHD1D} comes from
the divergence-free condition \eqref{eq:2D:BxBy0}.
Numerically preserving  \eqref{eq:2D:BxBy0}
is very non-trivial (for $d\ge 2$) but important for the robustness of numerical scheme,
and has to be respected.
In physics, numerically incorrect magnetic field topologies may lead to nonphysical plasma transport orthogonal to the magnetic field, see e.g.,  \cite{Brackbill1980}.
The condition \eqref{eq:2D:BxBy0} is also very crucial for the stability of induction equation \cite{Yang2016,Balsara2017}.
Existing numerical experiments in the non-relativistic MHD case indicated
that violating the divergence-free condition of magnetic field
may lead to numerical instability and nonphysical or inadmissible solutions \cite{Brackbill1980,Balsara2004,Rossmanith2006,Balsara2012}.
Up to now, many numerical treatments have been proposed to reduce such risk, see e.g.,  \cite{Evans1988,Toth2000,Balsara2004,Li2005,Balsara2009,Li2011,ZhaoTang2016,WuShu2018} and references therein.

The physically meaningful and admissible solutions of the RMHD system \eqref{eq:RMHD1D}
must satisfy the constraints such as $\rho>0$, $p >0$ and $v < c=1$ etc.
However, most of existing RMHD schemes do not always preserve those constraints,
even though they have been used to simulate some RMHD flows successfully.
There exists a large and long-standing risk of failure when a numerical scheme is applied to the RMHD problems with large Lorentz factor, low density or pressure, or strong discontinuity. This is
because once the negative density or pressure, or the superluminal fluid velocity is obtained, the eigenvalue of  Jacobian matrix or Lorentz factor become
imaginary, causing ill-posedness of the discrete problem and the break down of the codes.
Therefore, it is highly desirable to design
physical-constraints-preserving (PCP) numerical schemes, in the
sense of that the solutions of PCP schemes always
belong to the set of physically admissible states
\begin{equation}\label{eq:RMHD:definitionG}
{\mathcal G} := \left\{ \left. \vec U=(D,\vec m,\vec B,E)^{\top}\in {\mathbb R^{8}}~ \right| ~\rho(\vec U)>0,\ p(\vec U)>0,\ v(\vec U) <c = 1  \right\}.
\end{equation}
Because the functions $\rho(\vec U)$, $p(\vec U)$ and $v(\vec U)$ in \eqref{eq:RMHD:definitionG}   and $\vec F_i(\vec U)$ are highly nonlinear and cannot be explicitly formulated  in terms of $\vec U$,
it is extremely difficult to check  whether a given state $\vec U$ is admissible,
or  a numerical scheme is PCP.
For such a reason, developing the
PCP schemes for the RMHDs is highly challenging.

Recent years have witnessed some advances in
developing bound-preserving high-order accurate schemes for hyperbolic conservation laws.
Those schemes are mainly built on two types of limiting procedures.
One is the simple scaling limiting procedure
for the reconstructed or evolved solution polynomials in
a finite volume or discontinuous Galerkin (DG) method,
see e.g., \cite{zhang2010,zhang2010b,Xing2010,zhang2012,cheng,Zhang2017}.
Such a limiter has been shown to maintain the high-oder accuracy, see \cite{zhang2010,zhang2010b,Zhang2017}.
Another is the flux-corrected limiting procedure,
which can be used to high-order finite difference, finite volume
and DG methods, see e.g.,  \cite{Xu_MC2013,Hu2013,Liang2014,XiongQiuXu2014,Christlieb}.
A survey of the maximum-principle-satisfying or positivity-preserving
high-order schemes based on the first type limiter was presented in Ref.\ \cite{zhang2011b}.
The readers are also referred to Ref.\ \cite{xuzhang2016} for a review of
those two approaches.
The first work on PCP methods for relativistic hydrodynamics (RHD) was made in Ref.\ \cite{WuTang2015},
where the Lax-Friedrichs (LxF) scheme was rigorously proved to be PCP and
the PCP high-order accurate finite
difference WENO schemes {were} developed.
The bound-preserving {DG} methods were later extended from the non-relativistic
case \cite{zhang2010b} to the ideal special RHD case in Ref.\ \cite{Qin2016}.
More recently, the PCP high-order accurate central DG methods were proposed in Ref.\ \cite{WuTang2017ApJS} for the special RHDs with a general EOS \eqref{eq:EOS:h}.
Extension of the PCP methods from special to general RHDs is very nontrivial.
An earlier work \cite{Radice2014} attempted to construct the PCP scheme for the general RHDs, but only enforced the density
positivity. The importance and difficulty of designing completely PCP schemes were mentioned in Refs. \cite{Rezzolla2013,Radice2014}.
Very recently, the frameworks of designing provably PCP high-order accurate finite difference, finite volume and DG methods were established in Ref.\ \cite{Wu2017} for the general RHDs with a general EOS.
There was no work theoretically showed the PCP property of any numerical scheme of RMHDs
until the recent 
breakthrough in Ref.\ \cite{WuTangM3AS2017}.
With the sophisticated analysis on skillfully mining
the important mathematical properties of admissible state set,
the work \cite{WuTangM3AS2017} first developed several one- and two-dimensional PCP schemes
for RMHDs with the ideal EOS \eqref{eq:iEOS}, and  also revealed in theory for the first time that
the discrete divergence-free condition  is closely connected with
the PCP property of RMHD schemes.
In fact, it was also a blank in developing {\em provably} positive high-order schemes for the non-relativistic ideal compressible MHDs until the recent path-breaking work \cite{Wu2017b,WuShu2018}.
It is also noticed that,
for the incompressible
flow system in the vorticity-stream function formulation,
there is also a
divergence-free condition (but) on
fluid velocity, i.e., the incompressibility condition,  	
which is crucial
in designing
schemes that satisfies the
maximum principle of vorticity,
see e.g., \cite{zhang2010,LiXieZhang}.
An important difference in our RMHD case is that
our divergence-free quantity
(the magnetic field)
is also nonlinearly related to defining the admissible states, see \eqref{eq:RMHD:definitionG2}.

The ideal gas EOS with a constant adiabatic index is a poor approximation for most relativistic astrophysical flows, although it is commonly used in the RHDs and RMHDs.
The aim of this paper is to extend the theoretical analysis in Ref.\ \cite{WuTangM3AS2017}
to numerical schemes for the multi-dimensional RMHDs with the general EOS  \eqref{eq:EOS:h} on more general meshes (with non-self-intersecting polytopes).
Another purpose is to propose
a discrete divergence-free condition for such general case which is critical
for designing the PCP schemes.
In the case of a general EOS, one needs to
analytically handle the function $h(p,\rho)$ without a specific expression.
This introduces additional nonlinearity into the problem and poses additional difficulties
in studying the admissible state set. Moreover,
conducting the PCP analysis on a general mesh is also nontrivial and more complicated,
in comparison with that on the uniform Cartesian meshes considered in Ref.\ \cite{WuTangM3AS2017}.

The 
paper is organized as follows. Section \ref{sec:eqDef} extends
the properties of  $\mathcal G$ in Ref.\ \cite{WuTangM3AS2017}  for the general EOS \eqref{eq:EOS:h},
including two equivalent definitions of $\mathcal G$,
its convexity and  generalized LxF splitting properties. They play pivotal roles in analyzing the PCP property of  numerical methods with the LxF type flux for the RMHD equations \eqref{eq:RMHD1D}, see  Section \ref{sec:scheme},
where the PCP properties of multi-dimensional first- and high-order accurate schemes
are analyzed on a general mesh.
Section \ref{sec:con} concludes the paper with several remarks.

\section{Properties of admissible state set for a general EOS}\label{sec:eqDef}

This section studies the properties of admissible state set $\mathcal G$ for a general EOS \eqref{eq:EOS:h}.
%

\subsection{Equation of state}

The function $h(p,\rho)$ in \eqref{eq:EOS:h} must satisfy
\begin{equation}\label{eq:hcondition1}
h(p,\rho) \ge \sqrt{1+p^2/\rho^2}+p/\rho,
\end{equation}
as revealed by the relativistic kinetic theory \cite{WuTang2017ApJS}.

The properties of the admissible state set $\mathcal G$ are established in Ref.\ \cite{WuTangM3AS2017}
for the ideal EOS \eqref{eq:iEOS}.
Those properties can be extended to the case of a general EOS \eqref{eq:EOS:h} under some reasonable assumptions. In the following, we will show such extension and omit the derivations
that are the same as those for the ideal EOS case.

The paper focuses on the causal EOS and also assume that the fluid's coefficient of thermal expansion is positive, which is valid
for  most of compressible fluids, e.g., the gases.
If assume $h(p,\rho)$
is differentiable in $\mathbb R^+\times \mathbb R^+$, then the inequality
\begin{equation}\label{eq:gEOSC}
h\left(\frac1{\rho} - \frac{\pt h(p,\rho)}{\pt p} \right) < \frac{\pt h(p,\rho)}{\pt \rho} < 0,
\end{equation}
holds \cite{WuTang2017ApJS}.

The general EOS \eqref{eq:EOS:h} can also be expressed as
\begin{equation}
\label{eq:EOS:p:rhoh}
p = p(\rho,h),
\end{equation}
then the inequalities \eqref{eq:hcondition1} and \eqref{eq:gEOSC} respectively become
\begin{equation}\label{eq:hcondition11}
p(\rho,h) \le \frac{h^2-1}{2h} \rho,
\end{equation}
and
\begin{equation}\label{eq:gEOSd}
h \left( \frac1{\rho} \frac{\pt p(\rho,h)}{\pt h}  -1 \right) < - \frac{\pt p(\rho,h)}{\pt \rho} < 0, \quad \frac{\pt p(\rho,h)}{\pt h} >0.
\end{equation}
{The inequality \eqref{eq:gEOSC}
	or \eqref{eq:gEOSd} will only be used in Lemma  \ref{theo:RMHD:fUincrease}.}

Besides the conditions \eqref{eq:hcondition1} and \eqref{eq:gEOSC} or \eqref{eq:EOS:p:rhoh}, the paper also assumes that
\begin{equation} \label{eq:hpto1}
\mathop{\lim }\limits_{p \to 0^+ }  h( p,\rho) = 1,
\end{equation}
or equivalently,
\begin{equation}\label{eq:epto0}
\mathop{\lim }\limits_{h \to 1^+ } p(\rho,h)=0,
\end{equation}
for any fixed positive $\rho$. The above conditions  \eqref{eq:hcondition1}, \eqref{eq:gEOSC} and \eqref{eq:hpto1} are reasonable because they are satisfied by the ideal EOS \eqref{eq:iEOS} and most of the other EOS reported
in the numerical RHDs, see e.g.,  \cite{Mathews,Mignoneetal:2005,Ryu,WuTang2017ApJS}.

\begin{remark}
	The conditions \eqref{eq:gEOSC} and \eqref{eq:hpto1} imply
	\begin{equation*}
	\frac{\pt e(p,\rho)}{\pt p}>0,\quad \mathop{\lim }\limits_{p \to 0^+ }  e( p,\rho) = 0,
	\end{equation*}
	which further yield
	\begin{equation*}
	e(p,\rho) > \mathop{\lim }\limits_{ \delta_p \to 0^+ }  e( \delta_p,\rho) =0,
	\end{equation*}
	for any $p,\rho \in {\mathbb{R}}^+$. Therefore, the positivity of $e$ can be guaranteed if $p>0$ and $\rho>0$, and thus the physical constraints in \eqref{eq:RMHD:definitionG} do not need to include $e(\vec U)>0$ and $h({\bf U})>1$.
\end{remark}

\subsection{Nonlinearity and challenges}

The main challenges in studying $\mathcal G$ and the PCP property of numerical schemes   come from the intrinsic complexity and nonlinearity of \eqref{eq:RMHD1D}.
Especially the inherent strong nonlinearity is contained in several constraints in \eqref{eq:RMHD:definitionG}, because there is
no explicit expression of {$\rho(\vec U), p(\vec U)$, and $\vec v(\vec U)$ for the RMHDs,
	even for the ideal EOS  \eqref{eq:iEOS}.

	In practice, the values of $(\rho,p,{\bf v})$ should be derived from the given value of  $\vec U$} by solving
some  nonlinear algebraic equation, see e.g.,  \cite{Balsara2001,Zanna:2003,GodunovRMHD,MignoneHLLCRMHD,Newman,Noble}.
The present paper considers the following nonlinear algebraic equation (consistent with the one used in Ref.\ \cite{MignoneHLLCRMHD} for the ideal EOS)
\begin{equation}\label{eq:RMHD:fU(xi)}
f_{ \bf U}(\xi ) := \xi  - p\left( \frac{D}{W}, \frac{\xi}{D W} \right)  + {\left| \vec B \right|^2} - \frac{1}{2}\left[ {\frac{{{{\left| \vec B \right|}^2}}}{{{{W}^2}}} + \frac{{{{(\vec m \cdot \vec B)}^2}}}{{{\xi ^2}}}} \right] - E = 0,
\end{equation}
for the unknown $\xi\in \mathbb R^+$, where $p$ denotes the function $p(\rho,h)$ in \eqref{eq:EOS:p:rhoh}, and
the Lorentz factor $W$   has been expressed as a function
of $\xi$  by
\begin{equation}\label{eq:Wxi}
W(\xi) = \left(  {\xi^{-2}} {{(\xi + {|\vec B|^2})}^{-2}}
f_{\Omega}(\xi)\right)^{ - {1}/{2} },
\end{equation}
{with}
\begin{equation}\label{eq:Wxi-zzzzz}
f_{\Omega}(\xi):={\xi^2}{{(\xi + {|\vec B|^2})}^2} - \left[ {\xi^2}{|\vec m|^2} + (2\xi+{|\vec B|^2})  {{(\vec m \cdot \vec B)}^2} \right].
\end{equation}
It is reasonable to find the solution of  \eqref{eq:RMHD:fU(xi)} within the
interval
\begin{equation}\label{eq:Wxi-zzzzz000000000000}
\Omega_f:=\mathbb{R}^+ \cap
\left\{ {\left. \xi \right| f_{\Omega}(\xi) >0} \right\},
\end{equation}
otherwise,
$f_{\Omega}(\xi)\leq 0$ such that  $W(\xi)$ takes the value of  0 or  the imaginary number.
If denote  the solution of the equation \eqref{eq:RMHD:fU(xi)}  by $\xi_*=\xi_*(\vec U)$,
then $\xi_*=\rho(\vec U) h(\vec U)  W^2(\xi_*)=\rho(\vec U) h(\vec U) /\left(1-v^2(\vec U)\right)$,
and the values of the primitive variables $\rho(\vec U)$, $p(\vec U)$, and $v(\vec U)$ in \eqref{eq:RMHD:definitionG} can be calculated by
\begin{align}\label{eq:RMHD:getv}
&
\vec v(\vec U) = \left( {\vec m + \xi_* ^{ - 1}(\vec m \cdot \vec B) \vec B} \right)/(\xi_*+ {|\vec B|^2}),\\
&
\label{eq:RMHD:getrho}
\rho (\vec U) = \frac{D}{{W(\xi_*)}},\quad h(\vec U) = \frac{\xi_*}{D W(\xi_*)} , \\
\label{eq:RMHD:getp}
&
p(\vec U) = p\left( \rho (\vec U), h(\vec U) \right) .
\end{align}
The above procedure  clearly shows the strong nonlinearity of the functions
{$\vec v(\vec U)$, $\rho(\vec U)$}, and $p(\vec U)$,
as well as  the  challenges in verifying whether  $\vec U$ is in the set
$\mathcal G$.
As it is seen from \eqref{eq:RMHD:fU(xi)},
such nonlinearity is much stronger for a general EOS
in comparison with the ideal EOS.
Moreover, one needs to handle the function
$p(\rho,h)$ without a specific expression, leading to some difficulties
different from those in the ideal EOS case \eqref{eq:iEOS}.
To overcome the above challenges, 
two equivalent definitions of the admissible state set $\mathcal G$
will be given in the following.
The  first  is very suitable to check whether a given state $\vec U$  is admissible and construct
the PCP limiter for the developing high-order accurate robust schemes for the RMHDs,
while the second is very effective in verifying the PCP property of a numerical scheme.

\subsection{First equivalent definition}\label{sec:first}

This subsection introduces the first equivalent definition of the admissible state set $\mathcal G$.

\begin{lemma}\label{theo:RMHD:condition}
	The admissible state $\vec U=(D,\vec m,\vec B,E)^{\top} \in {\mathcal G}$ must satisfy
	\begin{equation}\label{eq:GconToG2}
	D>0,\quad  q(\vec U):= E-\sqrt{D^2+|\vec m|^2}>0.
	\end{equation}
\end{lemma}

\begin{proof}
	The proof is the same as that of Lemma 2.1 in Ref.\ \cite{WuTangM3AS2017} for the ideal EOS, except for
	using the condition
	\eqref{eq:hcondition1} instead of $\Gamma \in (1,2]$.
	Note that the following inequality is used
	$$
	(\rho h W^2 - p)^2 > |\rho h W^2 v|^2 + (\rho W)^2,
	$$
	which corresponds to $q(\vec U)>0$ in the RHD case and has been
	proved in Ref.\ \cite{WuTang2015} for ideal EOS and Ref.\ \cite{WuTang2017ApJS} for the general EOS under the condition \eqref{eq:hcondition1}.
	
\end{proof}

\begin{lemma}\label{theo:RMHD:CYcondition}
	$\vec U=(D,\vec m,\vec B,E)^{\top} \in {\mathcal G}$ if and only if
	$f_{\bf U}(\xi)$ has   unique {zero $\xi_*(\vec U)$ in $ \Omega_f$ and satisfies}
	\begin{equation}\label{eq:fourC}
	D>0,~q(\vec U) > 0,~ \xi_*(\vec U) >0,~  f_4 ( \xi_*(\vec U)  ) > 0,
	\end{equation}
	{where
		$f_4(\xi)$} is a
	quartic
	polynomial defined by
	\begin{equation}\label{eq:f4:xi}
	f_4(\xi):=  f_\Omega(\xi)  -{D^2}{({\xi} + {|\vec B|^2})^2} = (\xi + |\vec B|^2)^2 \left( \xi^2 W^{-2} (\xi) - D^2 \right).
	\end{equation}
\end{lemma}
\begin{proof}
	(i). Assume $\vec U \in {\mathcal G}$.
	Lemma \ref{theo:RMHD:condition} shows that the first two inequalities in \eqref{eq:fourC} hold.
	Because $\rho(\vec U)>0,p(\vec U)>0$ and $v(\vec U)<1$, one has
	$$\xi_*  =\rho hW^2 = \frac{\rho(\vec U) h (\vec U)}{1-v^2(\vec U)} \overset{\eqref{eq:hcondition1}}{\ge}  \frac{ \sqrt{\rho^2(\vec U)+p^2(\vec U)}+p (\vec U) }{1-v^2(\vec U)}  >0. $$
	On the other hand,  because $1< 1+{p(\vec U)}/{\rho (\vec U)} < h(\vec U) = \xi_*/ \big( D W(\xi_*) \big)$ and $v<1$, one has
	$\xi_* > D W(\xi_*) $,
	which implies $f_4 ( \xi_*  ) > 0$.

	(ii). Assume that the four inequalities in \eqref{eq:fourC} hold.
	Because  $D>0$ and  $\xi_*>0 $,
	one has
	$$ f_\Omega ( \xi_*  )  >f_\Omega(\xi_*)  -{D^2}{({\xi_*} + {|\vec B|^2})^2} = f_4 ( \xi_*  )>0 ,$$
	which implies
	$$
	W^{-2}=1 - v^2(\vec U) = \frac{  f_\Omega ( \xi_* ) } {\xi^2_*(\xi_* + |\vec B|^2)^2} > 0.
	$$
	Thus $v(\vec U)<1$ and $W(\xi_*)\ge 1$.
	Using \eqref{eq:RMHD:getrho} and $D>0$, one has $\rho(\vec U) = D/W(\xi_*) > 0$. Note that
	$$
	\xi^2_* W^{-2} (\xi_*) - D^2 = \frac{ f_4 ( \xi_*  ) } { (\xi_* + |\vec B|^2)^2 } >0,
	$$
	which yields $h(\vec U) = \xi_*/ \big( D W(\xi_*) \big) >1 $. Using \eqref{eq:gEOSd} and \eqref{eq:epto0} gives
	\begin{align*}
	p(\vec U) = p(\rho(\vec U),h(\vec U)) >  \mathop {\lim }\limits_{ h \to 1^+ } p(\rho(\vec U),h ) = 0.
	\end{align*}
	The proof is completed.
\end{proof}

%
%
%
%

\begin{lemma}\label{theo:RMHD:fUincrease}
	For any $\vec U=(D,\vec m,\vec B,E)^{\top}\in {\mathbb{R}}^{8}$ with $D>0$, the function $f_{ \bf U}(\xi ) $ defined in \eqref{eq:RMHD:fU(xi)} is strictly monotone increasing
	in the interval $\left(  \xi_4,+\infty  \right)$, and $\mathop {\lim }\limits_{\xi \to +\infty } f_{\bf U} (\xi) = + \infty$. Here $\xi_4=\xi_4({\bf U})$ is the unique positive root of $f_4(\xi)$ in $\Omega_f$.
\end{lemma}

\begin{proof}
	This result is nontrivial and the proof is also technical.
	From \eqref{eq:RMHD:fU(xi)} and \eqref{eq:Wxi},  the derivatives of $f_{ \bf U}(\xi ) $ and $W(\xi)$ with respect to $\xi $ is calculated as follows
	\begin{equation}\label{eq:dfUdxi}
	{f'_{\bf U}}(\xi ) =
	\frac{D}{W^2} W'(\xi) \frac{\pt p}{\pt \rho} \left( \frac{D}{W},\frac{\xi}{DW} \right) + \varXi_\xi \left[ 1 -   \frac{W }{D}   \frac{\pt p}{\pt h} \left( \frac{D}{W},\frac{\xi}{DW} \right)  \right] ,
	\end{equation}
	and
	\[
	W'(\xi ) =  - {W^3}\frac{{{{(\vec m \cdot \vec B)}^2}(3{\xi^2} + 3\xi{|\vec B|^2} + {|\vec B|^4}) + {|\vec m|^2}{\xi^3}}}{{{\xi^3}{{(\xi + {|\vec B|^2})}^3}}},
	\]
	where
	$$
	\varXi_\xi :=
	{1 + \frac{{{{\left|\vec  B \right|}^2}}}{{{W^3}}}W'(\xi )
		+ \frac{{{{(\vec m \cdot \vec B)}^2}}}{{{\xi^3}}}} = \frac{1}{{{W^2}}} - \frac{{\xi }}{{{W^3}}}W'(\xi ),
	$$
	has been used.
	

	Because $D/W(\xi)>0$ and $\xi/\big(D W(\xi)\big) > 1$ for any $\xi \in (\xi_4,+\infty)$, by using \eqref{eq:gEOSd} one obtains
	\begin{equation}\label{eq:gEOSieqProof}
	\frac{\pt p}{\pt \rho} \left( \frac{D}{W},\frac{\xi}{DW} \right) < \frac{\xi}{DW}\left[ 1-\frac{W}{D} \frac{\pt p}{\pt h} \left( \frac{D}{W},\frac{\xi}{DW} \right) \right],
	\end{equation}
	and
	\begin{equation}\label{eq:gEOSieqProof2}
	0< \frac{\pt p}{\pt h} \left( \frac{D}{W},\frac{\xi}{DW} \right) < \frac{D}{W}.
	\end{equation}
	Noting that $W'(\xi ) \le 0$ for any $\xi \in \Omega_f \subset (\xi_4,+\infty)$, and  using $\varXi_\xi >0$, \eqref{eq:gEOSieqProof} and \eqref{eq:gEOSieqProof2} give
	\begin{align*}
	{f'_{\bf U}}(\xi )
	&
	\ge
	\frac{\xi W'(\xi)}{W^3} \left[ 1-\frac{W}{D} \frac{\pt p}{\pt h} \left( \frac{D}{W},\frac{\xi}{DW} \right) \right] + \varXi_\xi \left[ 1 -   \frac{W }{D}   \frac{\pt p}{\pt h} \left( \frac{D}{W},\frac{\xi}{DW} \right)  \right]
	\\[2mm]
	&
	= \left( \varXi_\xi +  \frac{\xi W'(\xi)}{W^3}  \right) \left[ 1 -   \frac{W }{D}   \frac{\pt p}{\pt h} \left( \frac{D}{W},\frac{\xi}{DW} \right)  \right]
	\\[2mm]
	&
	=  \frac{1}{DW} \left[ \frac{D}{W} -      \frac{\pt p}{\pt h} \left( \frac{D}{W},\frac{\xi}{DW} \right)  \right] > 0,
	\end{align*}
	which implies that ${f_{\bf U}}(\xi ) $ is strictly monotone increasing
	in the interval $\left(  \xi_4,+\infty  \right)$.
	
	Let us prove $\mathop {\lim }\limits_{\xi \to +\infty } f_{\bf U} (\xi) = + \infty$. Using \eqref{eq:hcondition11} gives
	$$
	p\left( \frac{D}{W},\frac{\xi}{DW} \right) \le  \frac{ \xi^2W^{-2} - D^2 } {2\xi} ,
	$$
	which yields
	\begin{align*}
	{f_{\bf U}}(\xi ) & \ge \xi - \frac{ \xi^2W^{-2} - D^2 } {2\xi} + {\left| \vec B \right|^2} - \frac{1}{2}\left[ {\frac{{{{\left| \vec B \right|}^2}}}{{{{W}^2}}} + \frac{{{{(\vec m \cdot \vec B)}^2}}}{{{\xi ^2}}}} \right] - E \\
	& > \left(1- \frac{1}{2W^2} \right) \xi  - \frac{1}{2}\left[ {\frac{{{{\left| \vec B \right|}^2}}}{{{{W}^2}}} + \frac{{{{(\vec m \cdot \vec B)}^2}}}{{{\xi ^2}}}} \right] - E \to + \infty, \quad \mbox{as } \xi \to +\infty,
	\end{align*}
	where $\mathop {\lim }\limits_{\xi \to +\infty } W(\xi) = 1$ has been used. The proof is completed.
\end{proof}

%
%
%
%
%
%

\begin{remark}
	The proof and conclusion of Lemma
	\ref{theo:RMHD:fUincrease}
	also hold for
	the RHD case by taking ${\bf B}={\bf 0}$.
	In other words, they actually provide
	a different way to show  Lemma 3.2 in Ref.\ \cite{WuTang2017ApJS}.
	The current proof only requires the
	differentiability of $p(\rho,h)$, while the proof in
	Ref.\ \cite{WuTang2017ApJS} needs the
	continuously differentiability of $e(p,\rho)$.
	Note that, without any revision,
	the proofs of Lemmas \ref{theo:RMHD:CYcondition}
	and \ref{theo:RMHD:fUincrease} work under a slightly milder condition
	\begin{equation}\label{eq:hcondition1111}
	h(p,\rho) \ge 2p/\rho,
	\end{equation}
	than \eqref{eq:hcondition1} or equivalently \eqref{eq:hcondition11}.
	However, the condition \eqref{eq:hcondition1} or    \eqref{eq:hcondition11} is physically ensured, and can be directly deduced from the relativistic kinetic theory \cite{WuTang2017ApJS} without any assumption.
	Moreover, the condition \eqref{eq:hcondition1} is necessary
	and cannot be weaken for $q({\bf U})>0$,
	even in the case
	of ${\bf B}={\bf 0}$ (i.e., the RHD case).
	
\end{remark}

Based on the above lemmas and  Lemmas 2.3, 2.4, 2.6 and 2.7 in Ref.\ \cite{WuTangM3AS2017},
the
first equivalent definition of $\mathcal G$ can be established with the  proof similar
to that in Ref.\ \cite{WuTangM3AS2017} for the ideal EOS  and omitted here.

\begin{theorem}[First equivalent definition]\label{theo:RMHD:CYconditionFINAL2}
	The admissible state set~${\mathcal G}$ is equivalent to the   set
	\begin{align}
	{\mathcal G}_0 := \left\{   \vec U=(D,\vec m,\vec B,E)^{\top} \big|  D>0,q(\vec U)>0, \Psi (\vec U) > 0 \right\},
	\label{eq:RMHD:definitionG2}
	\end{align}
	where
	$$
	\Psi (\vec U) := \big( \Phi(\vec U)-2(|\vec B|^2-E) \big) \sqrt{\Phi(\vec U)+|\vec B|^2-E} - \sqrt{ \frac{27}{2} \bigg( D^2|\vec B|^2+(\vec m \cdot \vec B)^2 \bigg)},
	$$
	with ${\Phi(\vec U):}= \sqrt{({|\vec B|^2} - E)^2 + 3({E^2} - {D^2} - |\vec m|^2)}$.
\end{theorem}

\begin{remark}\label{rem:importantRemark}
	As pointed out in Ref.\ \cite{WuTangM3AS2017}, the constraint ${\Psi}(\vec U) >0$ in \eqref{eq:RMHD:definitionG2} is equivalent to two constraints $\hat q(\vec U)>0$ and $\tilde q(\vec U)>0$, where
	\begin{align*}
	\hat q(\vec U) &:= \sqrt{ \left( E-|\vec B|^2 \right)^2+3\left( E^2-D^2-|\vec m|^2 \right) } + 2 \left( E-|\vec B|^2 \right), \\
	\tilde q(\vec U) & := \Phi^6(\vec U) - \bigg( \left( E-|\vec B|^2 \right)^3 + \frac{27}{2} \left( {|\vec B|^2D^2 + |\vec m \cdot \vec B|^2} \right)
	\\
	& \quad - 9 \left( E^2-D^2-|\vec m|^2 \right) \left( E-|\vec B|^2 \right)  \bigg)^2.
	\end{align*}
	Checking those two constraints can be more effective
	in analytically showing that a numerical scheme is not PCP, see
	the proof of Theorem \ref{theo:disprove}.
\end{remark}

\subsection{Convexity and second equivalent definition}\label{sec:convexity}

As one can see from \eqref{eq:RMHD:definitionG2},
the admissible state set $\mathcal G_0=\mathcal G$ for a general EOS \eqref{eq:EOS:h}
is the same as that for the ideal EOS \eqref{eq:iEOS}.
This implies the fact that for
any given ${\bf U} \in \mathcal G_0=\mathcal G$,
even if ${\bf U}$ is the state of non-ideal gas,
there still exists a set of physical variables $(\rho,{\bf v},{\bf B},p)$ of an ideal gas such that the value of corresponding
conservative vector is  $\bf U$.
As a result, the convexity of $\mathcal G_0=\mathcal G$
and the second equivalent definition are directly followed from the analysis
in Ref.\ \cite{WuTangM3AS2017}.

\begin{theorem}\label{theo:RMHD:convex}
	The admissible state set ${\mathcal G}_0$ is a convex set.
\end{theorem}


The convexity of admissible
state set is very useful in the bound-preserving analysis, since it helps one reduce the
complexity if the schemes can be rewritten into a convex combination, see e.g.,  \cite{zhang2010b,WuTang2015,WuTang2017ApJS,Wu2017}.
Although
the last inequality in the proof of convexity
in Ref.\ \cite{WuTangM3AS2017} involves the ideal EOS,
according to
the fact figured out above,
it also holds for the general EOS. Certainly,
that inequality can also be directly derived for
a general EOS, same as
Lemma \ref{theo:RMHD:LLFsplit} with taking $\theta=0$.

\begin{theorem}[Second equivalent definition]\label{theo:RMHD:CYcondition:VecN}
	The admissible state set ${\mathcal G}$ or ${\mathcal G}_0$ is equivalent to the   set
	\begin{align}
	\nonumber
	{\mathcal G}_1 := \big\{   \vec U=(D,\vec m,\vec B,E)^{\top} \in \mathbb{R}^8 \big|  D>0, \vec U \cdot
	{{\vec n^*}} + {p^*_m} >0, \\
	\mbox{for any {${\vec B^*}, {\vec v^*} \in \mathbb{R}^3$} with  $  |\vec v^*|<1$} \big\},
	\label{eq:RMHD:CYcondition:VecNG1} \end{align}
	where
	\begin{align}\label{eq:RMHD:vecns}
	&{\vec n}^* = {\left( - \sqrt {1 - {|\vec v^*|}^2} ,~
		- {\vec v}^*,~ - (1 - {|\vec v^*|}^2) {\vec B}^* - ({\vec v}^* \cdot {\vec B}^*) {\vec v}^*,~1 \right)^{\top}},\\
	& p_{m}^*  = \frac{ (1-{|\vec v^*|}^2) |{\vec B}^*|^2 +({\vec v}^* \cdot {\vec B}^*)^2 }{2}. \label{eq:RMHD:vecns2}
	\end{align}
\end{theorem}

The importance of the second equivalent form \eqref{eq:RMHD:CYcondition:VecNG1} lies in that all constraints
are  linear  with respect to $\vec U$ so that
it will be very effective in theoretically verifying
the PCP property of  the numerical schemes for the RMHD equations \eqref{eq:RMHD1D}.

\begin{remark}
	Theorems \ref{theo:RMHD:CYconditionFINAL2} and \ref{theo:RMHD:CYcondition:VecN} indicate that
	${\mathcal G}={\mathcal G}_0={\mathcal G}_1$, which
	will not be deliberately distinguished henceforth.
\end{remark}

As a direct consequence of Theorem \ref{theo:RMHD:convex} or \ref{theo:RMHD:CYcondition:VecN}, the following corollary holds.

\begin{corollary}\label{lam:new:convex}
	If define  	
	\begin{align*}
	\overline{\mathcal G} := \big\{   \vec U=(D,\vec m,\vec B,E)^{\top} \in \mathbb{R}^8 \big|  D>0, \vec U \cdot
	{{\vec n^*}} + {p^*_m} \ge 0,
	\\
	\mbox{for any {${\vec B^*}, {\vec v^*} \in \mathbb{R}^3$} with  $  |\vec v^*|<1$} \big\},
	\end{align*}
	then
	$\lambda {\bf U} + (1-\lambda) \tilde{\bf U} \in {\mathcal G} $
	for any
	${\bf U} \in {\mathcal G}$, $\tilde{\bf U} \in \overline {\mathcal G}$ and  $\lambda \in (0,1]$.
\end{corollary}

Theorem \ref{theo:RMHD:CYcondition:VecN}  also implies the following
orthogonal invariance of the admissible state set  ${\mathcal G}_1$.

\begin{corollary}[Orthogonal invariance \cite{WuTangM3AS2017}] \label{lem:RMHD:zhengjiao}
	Let $\vec T :={\rm diag}\{1,\vec T_3,\vec T_3,1\}$, where
	$\vec T_3$ denotes any  orthogonal matrix  of size $d$.
	If $\vec U \in{\mathcal G}_1$, then
	$\vec T \vec U \in{\mathcal G}_1$.
\end{corollary}

\subsection{Generalized Lax-Friedrichs splitting properties} \label{sec:GLFs}

The section presents the generalized LxF splitting properties of the admissible state set $\mathcal G$. As revealed in Ref.\ \cite{WuTangM3AS2017}, the LxF splitting property
\begin{align}\label{eq:LxFprop}
\mbox{$\vec U \pm \alpha^{-1} \vec F_i (\vec U) \in {\mathcal G}$ for all  $\vec U \in {\mathcal G}$,~$\alpha \ge 1$,}
\end{align}
does not always hold for a nonzero magnetic field.
Therefore, we would
like to seek some alternative properties which are weaker than \eqref{eq:LxFprop}. By considering the
convex combination of some LxF splitting terms, we discover the generalized LxF splitting properties of $\mathcal G$ under some ``discrete divergence-free'' condition for the magnetic
field.
The following constructive inequality plays a pivotal role in establishing the generalized LxF splitting properties.

\begin{lemma}\label{theo:RMHD:LLFsplit}
	If $\vec U \in {\mathcal G}$, then
	for any $\theta \in [-1,1]$ and ${\vec B}^*$, $\vec v^*$ $\in \mathbb{R}^3$ with $|\vec v^*|<1$
	it holds
	\begin{equation}\label{eq:RMHD:LLFsplit}
	\big( \vec U + \theta \vec F_i(\vec U) \big) \cdot \vec n^* +  p_{m}^* + \theta \big( v_{i}^* p_{m}^* - B_i (\vec v^* \cdot \vec B^*)\big)\ge 0,
	\end{equation}
	where $i\in\{1,2,\cdots,d\}$, and $\vec n^*$ and $p_{m}^*$ are defined in \eqref{eq:RMHD:vecns} and \eqref{eq:RMHD:vecns2}, respectively.
\end{lemma}

\begin{proof}
	Without loss of generality, let us focus on the case of $i=1$, and show that
	\begin{equation}\label{eq:RMHD:LLFsplit-thz}
	{\mathcal H} (\rho,h,\vec v, \vec B,\vec v^*,\vec B^*,\theta ) := \big( \vec U + \theta \vec F_1(\vec U) \big) \cdot \vec n^* + (1+\theta v_{1}^* ) p_{m}^* - \theta B_1 (\vec v^* \cdot \vec B^*)>0.
	\end{equation}
	Note that ${\mathcal H}$ can be expressed as
	\begin{align*}
	{\mathcal H} & =  (1+\theta v_1)  \left( \rho h W^2 ( 1-  \vec {  v} \cdot {\vec v^*}  ) -\rho W (W^{-1})^* \right) - (1+\theta v_1^*) p(\rho,h)
	+ {\mathcal H}_0 (\vec v, \vec B,\vec v^*,\vec B^*,\theta )
	\\
	& \ge \min \{ H_+, H_- \} + {\mathcal H}_0 (\vec v, \vec B,\vec v^*,\vec B^*,\theta ),
	\end{align*}
	where ${\mathcal H}_0 (\vec v, \vec B,\vec v^*,\vec B^*,\theta ):= \mathop {\lim }\limits_{\rho \to 0^+ } \mathop {\lim }\limits_{h \to 1^+ } {\mathcal H} (\rho,h,\vec v, \vec B,\vec v^*,\vec B^*,\theta ) $, and
	\begin{align*}
	H_\pm &:=  (1 \pm v_1)  \Big( \rho h W^2 ( 1-  \vec {  v} \cdot {\vec v^*}  ) -\rho W (W^{-1})^* \Big) - (1 \pm v_1^*) p
	\\ & = (1\pm v_1) \rho h W^2 - p - \Big[  \Big( (1\pm v_1) \rho h W^2 v_1 \pm p \Big) v_1^*
	\\
	& \quad
	+ (1\pm v_1) \Big( \rho h W^2 v_2 v_2^*
	+ \rho h W^2 v_3 v_3^* + \rho W (W^{-1})^*  \Big)  \Big]
	\\
	& \ge (1\pm v_1) \rho h W^2 - p - \bigg( \Big( (1\pm v_1) \rho h W^2 v_1 \pm p \Big)^2
	\\
	& \quad + (1\pm v_1)^2 \Big( \rho^2 h^2 W^4 (v_2^2+v_3^2) + \rho ^2 W^2  \Big) \bigg)^{\frac12},
	\end{align*}
	in which the inequality follows from the Cauchy-Schwarz inequality, and $\sqrt{ |\vec v^*|^2 + (W^{-2})^*  }=1$ has been used. Note that
	\begin{align*}
	(1\pm v_1) \rho h W^2 - p
	&\ge (1 - |\vec v|) \left( \frac{\rho h}{1-|\vec v|^2} \right) - p = \frac{\rho h}{1+|\vec v|}  - p
	\\
	&> \frac{1}{2} \left(\rho h  - 2p\right) \overset{\eqref{eq:hcondition1}} {\ge} \frac{1}{2} \left( \sqrt{\rho^2+p^2}  - p\right) > 0,
	\end{align*}
	and
	\begin{align*}
	& \Big( (1\pm v_1) \rho h W^2 - p \Big)^2 - \bigg( \Big( (1\pm v_1) \rho h W^2 v_1 \pm p \Big)^2
	\\
	&+ (1\pm v_1)^2 \Big( \rho^2 h^2 W^4 (v_2^2+v_3^2) + \rho ^2 W^2  \Big) \bigg)
	\\
	& = (1\pm v_1)^2 W^2 \Big( (\rho h - p)^2 - (\rho^2+p^2) \Big) \overset{\eqref{eq:hcondition1}} {\ge} 0.
	\end{align*}
	Thus $H_\pm \ge 0$. Therefore, it only needs to show ${\mathcal H}_0 \ge 0$.
	Note that
	${\mathcal H}_0$
	only involves ${\bf v},{\bf B}$, and does not depend on the EOS.
	Therefore, the proof of ${\mathcal H}_0 \ge 0$ is, although very technical, but the same
	as the ideal EOS case in Ref.\ \cite{WuTangM3AS2017} and omitted here.
\end{proof}

Once the inequality \eqref{eq:RMHD:LLFsplit} is constructed,
the generalized LxF splitting properties are directly followed.
Here we present a general formulation of the property for studying
multi-dimensional PCP schemes on
a general  mesh.
The readers are referred to Ref.\ \cite{WuTangM3AS2017} for several special versions of the properties for
the cases of one- and multi-dimensional Cartesian meshes.

For any vector ${\bm \xi}=(\xi_1,\cdots,\xi_d) \in \mathbb{R}^d$, define the following inner products
\begin{equation}\label{eq:WKLnotation}
\big\langle {\bm \xi}, {\bf B} \big\rangle := \sum_{ \ell =1 }^d \xi_\ell B_\ell, \quad
\big\langle  {\bm \xi}, {\bf F} \big\rangle := \sum_{ \ell =1 }^d {\bf \xi}_\ell {\bf F}_\ell.
\end{equation}

\begin{theorem}[Generalized LxF splitting]\label{theo:RMHD:LLFsplit2Dus} 
	Assume that
	$$
	{\vec U}^{ij}=\left( D^{ij}, {\bf m}^{ij}, {\bf B}^{ij}, E^{ij} \right)^\top \in {\mathcal G}, \quad
	i=1,\cdots,Q,~j=1,\cdots,J,
	$$
	and satisfy the $d$-dimensional  ``discrete divergence free'' condition over a $d$-polytope
	(whose boundary does not intersect itself) with $J$ edges ($d=2$) or faces ($d=3$)
	\begin{equation}\label{eq:descrite2DDIVus}
	\sum\limits_{j = 1}^J {\left[ \sum\limits_{i = 1}^Q \omega _i
		\left\langle {\bm \xi}_j, {\bf B}^{ij} \right\rangle
		\right]} {\mathcal L}_j = 0,
	\end{equation}
	where 	
	${\mathcal L}_j>0$ and ${\bm \xi}_j=\left( \xi_{j}^{(1)},\cdots,\xi_{j}^{(d)} \right)$ are the $(d-1)$-dimensional Lebesgue measure and the unit outward normal vector of the $j$-th side or face of the polytope, respectively, and the sum of all positive numbers
	$\left\{\omega _i \right\}_{i=1}^Q$ equals one.
	Then for all $\alpha \ge c=1$ it holds
	\begin{align*}
	\bar{\vec U}: = \frac{1}{{\sum_{j = 1}^J {{\mathcal L}_j } }}\sum\limits_{j = 1}^J {\left[ {\sum\limits_{i = 1}^Q {\omega _i \Big( {{\vec U}^{ij}  - \alpha ^{ - 1}
					\left\langle {\bm \xi}_j, \vec F (\vec U^{ij} ) \right\rangle
				} \Big)} } \right]} {\mathcal L}_j  \in \overline{\mathcal G}.
	\end{align*}
\end{theorem}

Before the proof, we would like to provide
the specific meaning of the superscript $i,j$ on the vector ${\bf U}$, which will become clear in Section \ref{sec:scheme}.
The index $j$ represents the $j$-th edge or face of the polytope,
while $i$ stands for the $i$-th (quadrature) point on each edge or face. That is, ${\vec U}^{ij}$ denotes the approximate value of
${\bf U}$ at the $i$-th quadrature point on the $j$-th edge or face of the polytope, and $\omega_i$ corresponds to the associated quadrature weight at
that point.

\begin{proof}
	The conclusion for $d=2$ has been proved in Ref.\ \cite{WuTangM3AS2017}. The following proof is focused on the case of $d=3$. Using the spherical coordinates, one can represent the unit vector ${\bm \xi}_j$ as
	$$
	{\bm \xi}_j = ( \sin \theta_j \cos \varphi_j, \sin \theta_j \sin \varphi_j, \cos \theta_j  ),
	$$
	where the angles $ \theta_j \in [0,\pi]$ and $ \varphi_j \in [0,2\pi)$, $j=1,2,\cdots,J$.
	The rotational invariance property of the 3D RMHD equations \eqref{eq:RMHD1D} yields
	$$
	\xi_j^{(1)} \vec F_1 (\vec U^{ij} )  + \xi_j^{(2)} \vec F_2 (\vec U^{ij} )  + \xi_j^{(3)} \vec F_3 (\vec U^{ij} )  = \vec T^{-1}_j \vec F_1(\vec T_j {\vec U}^{ij}),
	$$
	where $\vec T_j:={\rm diag} \left\{1,\vec T_{3,j},\vec T_{3,j},1 \right\}$ with the rotational matrix $\vec T_{3,j}$ defined by
	
	$$
	{\vec T}_{3,j} :=  \begin{pmatrix}
	\sin \theta_j \cos \varphi_j~ & ~\sin \theta_j \sin \varphi_j~ & ~\cos \theta_j   \\
	-\sin\varphi_j~ & ~\cos \varphi_j~ & ~0 \\
	-\cos \theta_j \cos \varphi_j~ & ~-\cos \theta_j \sin \varphi_j~ & ~\sin \theta_j
	\end{pmatrix}.
	$$
	For each $j$ and any $\vec B^*,\vec v^*\in {\mathbb{R}}^3$ with $|\vec v^*|<1$, let $\hat {\vec v}^*= {\vec v}^* \vec T_{3,j}$ and $\hat {\vec B}^*= {\vec B}^* \vec T_{3,j}$. One has $|\vec {\hat v}^*| = |\vec v|<1,~\vec {\hat v}^* \cdot \vec {\hat B}^* = \vec v^* \cdot \vec B^*,~\hat p_m^* =   p_m^*$, and $\hat {\vec n}^* = \vec T_{j} {\vec n}^*$. Utilizing Lemma \ref{theo:RMHD:LLFsplit} for $\vec T_j \vec U^{ij}$, $\hat {\vec v}^*$,  and $\hat {\vec B}^*$ gives
	\begin{align} \nonumber
	0 & \le \left( \vec T_j \vec U^{ij} - \alpha^{-1}  \vec F_1(\vec T_j {\vec U}^{ij}) \right) \cdot \hat{\vec n}^* + \hat p_m^*
	\\ \nonumber
	& \quad- \alpha^{-1} \big( \hat v_1^* \hat p_m^* - \left(  \bm \xi_j \cdot \vec B^{ij} \right) (\hat {\vec v}^* \cdot \hat {\vec B}^* ) \big) \\ \nonumber
	& = \left(  \vec U^{ij} - \alpha^{-1} \vec T_j^{-1}  \vec F_1(\vec T_j {\vec U}^{ij}) \right) \cdot {\vec n}^* +  p_m^*
	\\
	&\quad  - \alpha^{-1} \big(  \left( \bm \xi_j \cdot {\vec v}^*  \right)  p_m^* - \left( \bm \xi_j \cdot \vec B^{ij} \right) ( {\vec v}^* \cdot  {\vec B}^* ) \big),\label{eq:wkl000}
	\end{align}
	where the orthogonality of $\vec T_j$ has been used. Hence, one has
	\begin{align*}
	\bar{\vec U} \cdot \vec n^* + p_{m}^*
	&  = \frac{1}{{\sum\limits_{j = 1}^J { {\mathcal L}_j } }}\sum\limits_{j = 1}^J {\left[ {\sum\limits_{i = 1}^{{Q}} {\omega _i \Big( \left( {{\vec U}^{ij}  - \alpha ^{ - 1} \vec T_j^{-1}  \vec F_1(\vec T_j {\vec U}^{ij})  } \right) \cdot  \vec n^* + p_{m}^*   \Big)    } } \right]} {\mathcal L}_j\\
	&  \overset{\eqref{eq:wkl000}} {\ge}
	\frac{1}{\alpha{\sum\limits_{j = 1}^J { {\mathcal L}_j } }}\sum\limits_{j = 1}^J {\left[ {\sum\limits_{i = 1}^{{ Q}} {\omega _i \Big(  \left( \bm \xi_j \cdot {\vec v}^* \right)  p_m^* - \left( \bm \xi_j \cdot \vec B^{ij} \right) ( {\vec v}^* \cdot  {\vec B}^* )  \Big) } } \right]} {\mathcal L}_j\\
	&  \overset{\eqref{eq:descrite2DDIVus}} {=}   \frac{p_m^*}{\alpha{\sum\limits_{j = 1}^J { {\mathcal L}_j } }}\sum\limits_{j = 1}^J {\left[ {\sum\limits_{i = 1}^{{Q}} {\omega _i   \left( \bm \xi_j \cdot {\vec v}^*  \right)    } } \right]} {\mathcal L}_j
	=  \frac{p_m^*}{\alpha{\sum\limits_{j = 1}^J { {\mathcal L}_j } }}\sum\limits_{j = 1}^J    \left( \bm \xi_j \cdot {\vec v}^*  \right)  {\mathcal L}_j =0,
	\end{align*}
	which implies that {$\bar{\vec U}$} satisfies the second constraint in  $\overline{\mathcal G}$.
	On the other hand,  $\bar{\vec U}$  satisfies the first constraint in $\overline{\mathcal G}$  because
	\begin{align*}
	& \frac{1}{{\sum\limits_{j = 1}^J { {\mathcal L}_j } }}\sum\limits_{j = 1}^J {\left[ {\sum\limits_{i = 1}^{{ Q}} {\omega _i D^{ij} \Big( { 1 - \alpha ^{ - 1} \big( \bm \xi_j \cdot {\vec v}^{ij} \big)} \Big)} } \right]} {\mathcal L}_j
	\\
	&
	\quad \ge
	\frac{1}{{\sum\limits_{j = 1}^J { {\mathcal L}_j } }}\sum\limits_{j = 1}^J {\left[ {\sum\limits_{i = 1}^{{ Q}} {\omega _i D^{ij} \Big(  1 - \alpha ^{ - 1}  |{\vec v}^{ij}|   \Big)} } \right]} {\mathcal L}_j  >0.
	\end{align*}
	The proof is completed.
\end{proof}

\section{Physical-constraints-preserving schemes}\label{sec:scheme}
This section applies the above properties of the admissible state set
$\mathcal G$ to  exploring and analyzing the PCP schemes for the $d$-dimensional ($d\ge 2$)
special RMHD equations \eqref{eq:RMHD1D}.
For the sake of convenience, the symbols $\bf x$  will be used to represent
the independent variables $(x_1,x_2,\cdots,x_d)$ in \eqref{eq:RMHD1D}.
It is worth noting that, {once the above properties of $\mathcal G$ are established},
the PCP analysis
will not directly involve the EOS.
{For such a reason, the 1D analysis
	in Ref.\ \cite{WuTangM3AS2017} can directly apply to
	the 1D case for a general EOS.}

Assume that the $d$-dimensional spatial domain is divided into a mesh ${\mathcal M}$ with cells $\{ {\mathcal I}_k\}$, which can be
general (non-self-intersecting) $d$-dimensional polytopes.
Moreover, the mesh can be unstructured. Let ${\mathcal N}_k$ denote the index set of all the neighboring cells of $ {\mathcal I}_k$.
For each $j \in {\mathcal N}_k$, let ${\mathcal E}_{kj}$ be the boundary of $ {\mathcal I}_k$ sharing with its neighboring cell $ {\mathcal I}_j$, i.e.,
${\mathcal E}_{kj} =  \partial{\mathcal I}_k \cap \partial{\mathcal I}_j$,
and $\bm  {\xi}_{kj}= \big( \xi_{kj}^{(1)}, \cdots, \xi_{kj}^{(d)} \big)$ be the unit normal vector
of ${\mathcal E}_{kj}$ pointing from ${\mathcal I}_k$ to ${\mathcal I}_j$.
The time interval is also divided into mesh $\{t_0=0, t_{n+1}=t_n+\Delta t_{n}, n\geq 0\}$
with the time step-size $\Delta t_{n}$ determined by the CFL-type condition.
Let $\bar{\vec U}_k^n$ be the numerical approximation to the cell-averaged value of $\vec U$
over the cell ${\mathcal I}_k$ at $t=t_n$. We aim at seeking numerical schemes whose solutions
$\{\bar{\vec U}_k^n\}$ always belong to  the admissible state set $\mathcal G$.

\subsection{First-order accurate scheme}\label{sec:FirstOrder}

Consider the following first-order accurate scheme on the mesh ${\mathcal M}$  
\begin{equation}\label{eq:1stscheme}
\bar{\vec U}_{k}^{n+1}
= \bar{\vec U}_{k}^{n} - \frac{\Delta t_n}{|{\mathcal I}_k|}   \sum_{j \in {\mathcal N}_k}
|{\mathcal E}_{kj}| \hat{\vec F}_{kj},
\end{equation}
where
$|{\mathcal I}_{kj}|$ and $|{\mathcal E}_{kj}|$ denote the $d$- and ($d-1$)-dimensional Hausdorff measures of ${\mathcal I}_k$ and ${\mathcal E}_{kj}$ respectively. The numerical flux $\hat{\vec F}_{kj}$ in \eqref{eq:1stscheme} is taken as the LxF type flux
\begin{equation}\label{eq:LFflux}
\hat{\vec F}_{kj} = \hat{\vec F} (  \bar{\vec U}_k^n , \bar{\vec U}_j^n  ; {\bm \xi}_{kj} ) := \frac12 \left(  \Big\langle
{\bm \xi}_{kj},  {\vec F} ( \bar{\vec U}_k^n )
+ {\vec F} ( \bar{\vec U}_j^n ) \Big\rangle - \alpha ( \bar{\vec U}_j^n - \bar{\vec U}_k^n ) \right),
\end{equation}
where $\alpha$ is an appropriate upper bound of the spectral radius of Jacobian matrix
$\sum_{\ell=1}^d \xi_{kj}^{(\ell)} \partial {\vec F}_\ell ({\vec U})/ \partial {\vec U} $
and can be taken as $\alpha \ge c = 1$. If $d=2$ and the mesh is a uniform rectangular mesh,
then the scheme \eqref{eq:1stscheme} becomes the 2D first-order scheme in Ref.\ \cite{WuTangM3AS2017}.

If $\bar {\vec U}_{k}^n$  belongs to $\mathcal G$ for all $k$,
but the magnetic field $\bar{\vec B}_{k}^n$ is not divergence-free
in some discrete sense,
then the solution $\bar {\vec U}_{k}^{n+1}$ of \eqref{eq:1stscheme}
does not always belong to $\mathcal G$, 
see the following theorem.

\begin{theorem} \label{theo:disprove}
	For a general EOS and any given mesh $\mathcal M$, under the CFL condition
	\begin{equation}\label{eq:CFL:LF2D000}
	0<    \frac{ \alpha \Delta t_n }{2 |{\mathcal I}_k|  } \sum_{j \in {\mathcal N}_k} |{\mathcal E}_{kj}|  <  1,
	\end{equation}
	there always exists a set of admissible states $\{\bar{\vec U}_{k}^{n},\forall k\}$ such that
	the solution $ \bar {\vec U}_{k}^{n+1}$ of \eqref{eq:1stscheme} does not belong to ${\mathcal G}$. In other words, the admissibility of $\bar{\vec U}_{k}^{n}\in {\mathcal G}, \forall k$, does not in general guarantee that $ \bar {\vec U}_{k}^{n+1} \in {\mathcal G}$, $\forall k$.
\end{theorem}

\begin{proof}
	It is proved by contradiction. Assume that
	$\bar{\vec U}_{k}^{n}\in {\mathcal G}$, $\forall k$, always
	ensure $ \bar {\vec U}_{k}^{n+1} \in {\mathcal G}$, $\forall k$.
	Without loss of generality, we take $\alpha=c=1$.
	For any $\epsilon,\tau >0$, define the following admissible primitive variables
\begin{align*}
	& \hat {\bf V} := (\hat \rho, \hat {\bf v},\hat{ \bf B}, \hat p)^\top = (\epsilon,0.5,0,0,0,0,0,\tau)^\top,
	\\
	&
	\tilde {\bf V} := (\tilde \rho, \tilde {\bf v},\tilde{ \bf B}, \tilde p)^\top = (\epsilon,0.5,0,0,1,0,0,\tau)^\top,
\end{align*}
	and	let $\hat {\bf U} := {\bf U}(\hat{\bf V}) \in {\mathcal G}$
	and $\tilde {\bf U} := {\bf U}(\tilde{\bf V}) \in {\mathcal G}$ be the corresponding conservative vectors, and ${\mathcal I}_{j_*}$ be a neighboring cell of ${\mathcal I}_k$, i.e., $j_* \in {\mathcal N}_k$.
	

	The rotational invariance property of $d$-dimensional RMHD equations \eqref{eq:RMHD1D} yields
	$$
	\Big\langle {\bm \xi}_{kj_*}, \vec F (\vec U ) \Big\rangle  = \vec T^{-1} \vec F_1(\vec T {\vec U}),
	$$
	for any ${\bf U}\in \mathcal {G}$,
	where  the matrix $\vec T:={\rm diag} \left\{1,\vec T_{3},\vec T_{3},1 \right\}$,
	with the rotational matrix ${\vec T}_{3} $  defined by
	$$
	{\vec T}_{3} :=  \begin{pmatrix}
	\cos \varphi~ & ~\sin \varphi~ & ~0  \\
	-\sin\varphi~ & ~\cos \varphi~ & ~0 \\
	0~ & ~0~ & ~1
	\end{pmatrix},
	$$
	for $d=2$,
	where $( \cos \varphi, \sin \varphi )$ is the polar coordinates representation of ${\bm \xi}_{kj_*} $,
	or
	$$
	{\vec T}_{3} :=  \begin{pmatrix}
	\sin \theta \cos \varphi~ & ~\sin \theta \sin \varphi~ & ~\cos \theta   \\
	-\sin\varphi~ & ~\cos \varphi~ & ~0 \\
	-\cos \theta \cos \varphi~ & ~-\cos \theta \sin \varphi~ & ~\sin \theta
	\end{pmatrix},
	$$
	for $d=3$, where $( \sin \theta \cos \varphi, \sin \theta \sin \varphi, \cos \theta  )$
	denotes the spherical coordinates representation of ${\bm \xi}_{kj_*}$.
	
	Consider the following special data
	$$
	\bar {\vec U}_{j_*}^{n} = {\bf T}^{-1} \tilde {\bf U}, \qquad \bar {\vec U}_{j}^{n} = {\bf T}^{-1}  \hat {\bf U},~\forall j \neq j_*,
	$$
	which are all admissible thanks to Corollary \ref{lem:RMHD:zhengjiao}.
	Substituting them into \eqref{eq:1stscheme} gives
	\begin{align*}
	\bar {\vec U}_{k}^{n+1} (\tau,\epsilon) &=
	{\bf T}^{-1} \hat {\bf U} - \theta \Big(
	\left \langle {\bm \xi}_{kj_*},
	{\bf F}_1 ( {\bf T}^{-1} \tilde {\bf U} ) - {\bf F} ( {\bf T}^{-1} \hat {\bf U} ) \right \rangle + {\bf T}^{-1} \hat{\bf U} - {\bf T}^{-1} \tilde {\bf U} \Big)
	\\
	& = {\bf T}^{-1} \Big( \hat {\bf U} - \theta \big( {\bf F}_1 ( \tilde {\bf U} ) - {\bf F}_1 ( \hat {\bf U} ) + \hat{\bf U} - \tilde {\bf U} \big) \Big)
	\\
	& = {\bf T}^{-1}
	\bigg( \frac{2\sqrt{3} \epsilon }{3},~\frac{2 \epsilon h(\tau,\epsilon)}{3}
	+ \frac{\theta}{2},~0,~0,~\theta,~0,~0,~\frac{4\epsilon h(\tau,\epsilon)}{3} - \tau + \frac{\theta}{2} \bigg)^\top,
	\end{align*}
	where $\theta := \frac{ \alpha \Delta t_n }{2 |{\mathcal I}_k|  }  |{\mathcal E}_{kj_*}|
	\in (0,\frac12) $ under the condition \eqref{eq:CFL:LF2D000}.
	By the assumption, one has $\bar {\vec U}_{k}^{n+1} (\tau,\epsilon) \in {\mathcal G}$ for all $\epsilon,\tau >0$.
	This yields $ {\bf T} \bar {\vec U}_{k}^{n+1} (\tau,\epsilon) \in {\mathcal G}$  by Corollary \ref{lem:RMHD:zhengjiao}.
	It then follows from Theorem \ref{theo:RMHD:CYconditionFINAL2} and {Remark \ref{rem:importantRemark}} that
	$$\tilde q( {\bf T} \bar {\vec U}_{k}^{n+1} (\tau,\epsilon) ) > 0,\quad \mbox{for all}~\epsilon,\tau >0.$$
	The continuity of $\tilde q({\bf U})$ with respect to ${\bf U}$ further implies
\begin{align*}
	0 & \le \mathop {\lim }\limits_{ \epsilon \to 0^+ }  \mathop {\lim }\limits_{ \tau \to 0^+ }
	\tilde q \left( {\bf T}  \bar {\vec U}_{k}^{n+1} (\tau,\epsilon) \right)
	\\
	&
	= \tilde q \left( \mathop {\lim }\limits_{ \epsilon \to 0^+ }  \mathop {\lim }\limits_{ \tau \to 0^+ } {\bf T}  \bar {\vec U}_{k}^{n+1} (\tau,\epsilon) \right)
	= \frac{ 27 \theta^7 (4 \theta + 1)^2 }{64} \times ( \theta -2 ) < 0,
\end{align*}
	which is a contradiction. Hence the assumption does not hold. The proof is completed.
\end{proof}

Let  $\bar {\vec U}_{k}^n =:(\bar  D_{k}^n, \bar { \vec m}_{k}^n, \bar  {\vec B}_{k}^n, \bar  E _{k}^n  )^\top$ and
$
\overline{\bf B}_{kj}^n  := \frac{1}{2} \Big( \bar{ \bf  B}_{k}^n + \bar{ \bf  B}_{j}^n \Big).
$
If the states $\{\bar {\bf U}_{k}^n\}$ are all admissible and
satisfy the following discrete divergence-free (DDF) condition
\begin{equation}\label{eq:DisDivB}
\mbox{\rm div} _{k} \bar {\bf B}^n :=
\sum\limits_{ j \in {\mathcal N}_k}
\Big\langle {\bm \xi}_{kj}, \overline{\bf B}_{kj}^n \Big\rangle
\big| {\mathcal E}_{kj} \big| = 0,
\end{equation}
then one can rigorously prove that the following conclusion
by using the generalized LF splitting property in Theorem \ref{theo:RMHD:LLFsplit2Dus}.

\begin{theorem} \label{theo:2DRMHD:LFscheme}
	If  $\bar {\vec U}_{k}^n \in {\mathcal G}$ and satisfies the DDF condition
	\eqref{eq:DisDivB}
	for all $k$, then under the CFL type condition \eqref{eq:CFL:LF2D000},
	the solution $ \bar {\vec U}_{k}^{n+1}$ of \eqref{eq:1stscheme}  belongs to ${\mathcal G}$ for all $k$.
\end{theorem}

\begin{proof}
	Using the identity
	\begin{equation}\label{eq:identity-WKL}
	\sum_{ j \in {\mathcal N}_k }
	\left\langle  \bm  {\xi}_{kj} , {\bf Z} \right\rangle \big| {\mathcal E}_{kj} \big|
	=  \int_{{\mathcal I}_k}  \sum_{\ell=1}^d \frac{\partial  Z_\ell }{\partial x_\ell}     {\rm d} {\bf x}  \equiv 0,
	\end{equation}
	for any constant vector ${\bf Z}=(Z_1,Z_2,Z_3) \in \mathbb{R}^3$, one can rewrite the scheme \eqref{eq:1stscheme} as
	\begin{equation}\label{eq:LFrew}
	\bar { \bf  U }_k^{n+1}  = \lambda_k {\bf \Xi} +
	( 1 - \lambda_k ) \bar {\bf  U}_k^{n},
	\end{equation}
	where 	$\lambda_k :=  \frac{\alpha \Delta t_n }{2|{\mathcal I}_k|}  \sum_{ j \in {\mathcal N}_k }
	| {\mathcal E}_{kj} | \in(0,1) $ under the condition \eqref{eq:CFL:LF2D000}, and
	\begin{equation*}
	{\bf \Xi} := \frac{1}{ \sum_{ j \in {\mathcal N}_k }
		| {\mathcal E}_{kj} |  }
	\sum_{ j \in {\mathcal N}_k }
	\bigg(
	\bar{ \bf  U }_j^n - \alpha^{-1}
	\left\langle
	{\bm \xi}_{kj}, {\vec F} ( \bar{\vec U}_j^n ) \right\rangle
	\bigg) | {\mathcal E}_{kj} |.
	\end{equation*}
	It also follows from the identity \eqref{eq:identity-WKL} that
	\begin{equation*}
	\sum\limits_{ j \in {\mathcal N}_k}
	\left\langle \bm \xi_{kj},\bar{\bf B}_{k}^n \right\rangle \big| {\mathcal E}_{kj} \big| = 0,
	\end{equation*}
	which implies that the DDF condition \eqref{eq:DisDivB} is equivalent to
	\begin{equation}\label{eq:DisDivB2}
	{\mbox{\rm div}}_{k}^{\mbox{\tiny \rm out}} \bar {\bf B}^n :=  \sum\limits_{ j \in {\mathcal N}_k}
	\left\langle \bm \xi_{kj},\bar{\bf B}_{j}^n \right\rangle \big| {\mathcal E}_{kj} \big| = 0,
	\end{equation}	
	Thanks to Theorem \ref{theo:RMHD:LLFsplit2Dus}, one has ${\bf \Xi} \in \overline{\mathcal G}$
	under the condition \eqref{eq:DisDivB} or \eqref{eq:DisDivB2}.
	Therefore the form \eqref{eq:LFrew} is a convex combination of $\bf \Xi \in \overline{\mathcal G}$ and $\bar{\bf U}_k^n \in {\mathcal G}$.
	The proof is completed by Corollary \ref{lam:new:convex}.
\end{proof}

A remark is given on the proof of Theorem \ref{theo:2DRMHD:LFscheme} and the following theorems.
\begin{remark}
	Taking Theorem \ref{theo:2DRMHD:LFscheme} as an example,
	an alternative {(equivalent)}
	presentation of the proof is as follows.
	First   show	$\bar { D }_k^{n+1} > 0$ directly, and then
	prove  $\bar{\vec U}_k^{n+1} \cdot \vec n^* + p_{m}^*>0$ for any ${\vec B}^*$, $\vec v^*$ $\in \mathbb{R}^3$ with $|\vec v^*|<1$. The proof of the second part can be done by Lemma \ref{theo:RMHD:LLFsplit}, and is almost the same as the proof of
	corresponding generalized LxF splitting property.
	%
\end{remark}

If $d=2$ and the mesh $\mathcal M$ consists of uniform rectangles, then the DDF condition
\eqref{eq:DisDivB2} becomes that defined in Ref.\ \cite{WuTangM3AS2017}.
On a general mesh, the scheme \eqref{eq:1stscheme} does not always preserve the condition \eqref{eq:DisDivB} or \eqref{eq:DisDivB2}.
However, in some special cases such as $\mathcal M$ consisting of uniform rectangles ($d=2$)
or cuboids ($d=3$), one can show the following result, see Ref.\ \cite{WuTangM3AS2017} for the case of $d=2$.

\begin{prop} \label{theo:2DDivB:LFscheme}
	Assume that the mesh $\mathcal M$ is a uniform Cartesian mesh.
	For the LxF scheme \eqref{eq:1stscheme},
	the divergence error
	$  \max\limits_{k} \left| {\rm div}_{k} \bar {\bf B}^{n} \right|$  
	does not grow with $n$
	under the condition \eqref{eq:CFL:LF2D000}.
	Moreover, 
	$\{\bar {\bf U}_{k}^n\}$ 
	satisfy \eqref{eq:DisDivB}
	for all $k$ and $n \in \mathbb{N}$ if \eqref{eq:DisDivB} holds for the discrete initial data $\{\bar {\bf U}_{k}^0\}$.
\end{prop}

\subsubsection{High-order accurate  schemes}

We then discuss the PCP high-order accurate schemes for the {$d$-dimensional} RMHD equations \eqref{eq:RMHD1D}. For the sake of convenience, the analysis will be focused on the time discretization using the forward Euler method,
and also work for the high-order accurate SSP ({strong stability preserving}) time discretization (cf. \cite{Gottlieb2009}).

Assume that  the approximate solution $\vec U_{k}^n ({\bf x})$ at time $t=t_n$ within the cell ${\mathcal I}_{k}$ is either reconstructed in the finite volume methods
from the cell average values $\{\bar{\vec U}_{k}^n\}$ or evolved in the discontinuous Galerkin (DG)  methods. The function $\vec U_{k}^n ({\bf x})$ is a vector of the polynomial of degree $ K$,  and its cell average value over the cell ${\mathcal I}_{k}$ equals $\bar {\vec U}_{k}^{n}$.

For the  RMHD equations \eqref{eq:RMHD1D},
the finite volume scheme or discrete equation for the cell average value in the DG method on
the mesh $\mathcal M$ may be written as
\begin{align}
\bar{\vec U}_{k}^{n+1}
= \bar{\vec U}_{k}^{n} - \frac{\Delta t_n}{|{\mathcal I}_k|}   \sum_{j \in {\mathcal N}_k}
|{\mathcal E}_{kj}| \hat{\vec F}_{kj},
\label{eq:2DRMHD:cellaverage}
\end{align}
where
\begin{equation}\label{eq:quadfulx}
\begin{aligned}
\hat{\vec F}_{kj} & = \sum_{ \mu =1 }^Q \omega_\mu
\hat{\vec F} (  {\vec U}_k^n ({\bf x}_{kj}^{(\mu)}) , {\vec U}_j^n ({\bf x}_{kj}^{(\mu)}); {\bm \xi}_{kj} )
\\
&\approx \frac{1}{|{\mathcal E}_{kj}|} \int_{{\mathcal E}_{kj}} \hat{\vec F} (  {\vec U}_k^n ({\bf x}) , {\vec U}_j^n ({\bf x}); {\bm \xi}_{kj} ) ds,
\end{aligned}
\end{equation}
and the numerical flux $\hat{\bf F}$ is taken
as the LxF flux defined in \eqref{eq:LFflux}.
Here
$\{ {\bf x}_{kj}^{(\mu)} \}$ are the $Q$-point
quadrature nodes on ${\mathcal E}_{kj}$, and $\{\omega_\mu\}$ are the associated weights satisfying $\sum_{\mu=1}^Q \omega_\mu = 1$.
In practice, the quadrature rule should meet certain accuracy requirement,
and is usually exact for polynomials of degree up to $K$.

Let ${\vec U}_{k}^n({\bf x})=:\big( D_{k}^n({\bf x}), \vec m_{k}^n({\bf x}),\vec B_{k}^n({\bf x}), E_{k}^n({\bf x})\big)^{\top}$ and
$$
\overline{\bf B}_{kj}^{n,(\mu)}  := \frac{1}{2} \Big(
{\bf B}_k^n ({\bf x}_{kj}^{(\mu)})+ {\bf B}_j^n ({\bf x}_{kj}^{(\mu)}) \Big),$$
and define the discrete divergences of $\{ {\bf B}_k^n({\bf x}) \}$ by
\begin{equation}\label{eq:DisDivB:cst}
\mbox{\rm div} _{k}  {\bf B}^n :=
\sum\limits_{ j \in {\mathcal N}_k}  \left[
\sum\limits_{ \mu=1}^Q \omega_\mu
\left\langle \bm \xi_{kj}, \overline{\bf B}_{kj}^{n,(\mu)} \right\rangle
\right]\big| {\mathcal E}_{kj} \big| ,
\end{equation}
which is an approximation to the left-hand side of
$$
\sum\limits_{ j \in {\mathcal N}_k} \int_{{\mathcal E}_{kj}}
\big\langle \bm \xi_{kj}, {\bf B} ({\bf x}) \big\rangle {\rm d} s
= \int_{{\mathcal I}_{k}}  \sum_{ \ell =1 }^d\frac{\partial B_\ell } {\partial x_\ell}    {\rm d} {\bf x} = 0.
$$
It is noticed that
\begin{equation}\label{eq:2eq1}
\mbox{\rm div} _{k} {\bf B}^n = \frac12 \left( \mbox{\rm div} _{k}^{\mbox{\tiny \rm in}} {\vec B}^n
+ \mbox{\rm div} _{k}^{\mbox{\tiny \rm out}} {\vec B}^n \right),
\end{equation}
where
\begin{align}
&
\mbox{\rm div} _{k}^{\mbox{\tiny \rm in}} {\vec B}^n :=
\sum\limits_{ j \in {\mathcal N}_k}  \left[
\sum\limits_{ \mu=1}^Q \omega_\mu
\left\langle \bm \xi_{kj}, {\bf B}_k^n ({\bf x}_{kj}^{(\mu)}) \right\rangle
\right]\big| {\mathcal E}_{kj} \big|,
\\ \label{eq:divoutH}
&
\mbox{\rm div} _{k}^{\mbox{\tiny \rm out}} {\vec B}^n :=
\sum\limits_{ j \in {\mathcal N}_k}  \left[
\sum\limits_{ \mu=1}^Q \omega_\mu
\left\langle \bm \xi_{kj}, {\bf B}_j^n ({\bf x}_{kj}^{(\mu)}) \right\rangle
\right]\big| {\mathcal E}_{kj} \big|.
\end{align}
The quantity ${\rm div}^{\rm out}_{k} {\bf B}^{n}$ defined in \eqref{eq:divoutH}
is consistent with the one given in \eqref{eq:DisDivB2} for the first-order scheme.  When $K=0$, the ``high-order'' version in \eqref{eq:divoutH}  reduces to the ``first-order'' version in \eqref{eq:DisDivB2}. In the later text, ${\rm div}^{\rm out}_{k} {\bf B}^{n}$ will be referred to as the general definition in \eqref{eq:divoutH}.

Similar to Theorem 3.5 in Ref.\ \cite{WuTangM3AS2017},
we have the following sufficient conditions for that
the high-order accurate scheme \eqref{eq:2DRMHD:cellaverage} is PCP on the mesh $\mathcal M$.

\begin{theorem} \label{thm:PCP:2DRMHD}
	If the polynomial vectors $\{{\vec U}_{k}^n({\bf x})\}$ satisfy
	\begin{gather}\label{eq:DivB:cst}
	\mbox{\rm div} _{k}  {\bf B}^n=0, \quad \forall k,
	\\
	\label{eq:AdmiCon}
	{\vec U}_k^n ({\bf x}_{kj}^{(\mu)}) \in {\mathcal G}, \quad \forall \mu \in \{1,\cdots,Q\},~\forall j \in {\mathcal N}_k,~\forall k,
	\end{gather}
	and  there exist a constant $ \beta_k \in \left(0,\frac12\right)$ for all $k$ such that
	\begin{equation}\label{eq:AdmiCon2}
	{\bf W}_k^n:=\frac{1}{1-2\beta_k}\left[\bar{\bf U}_k^n - \frac{ 2\beta_k }{ \sum_{j \in {\mathcal N}_k} |{\mathcal E}_{kj}| }
	\sum_{j \in {\mathcal N}_k} |{\mathcal E}_{kj}| \bigg( \sum_{ \mu =1 }^Q \omega_\mu {\bf U}_k^n ({\bf x}_{kj}^{(\mu)}) \bigg) \right] \in {\mathcal G},
	\end{equation}
	then under the CFL type condition
	\begin{equation}\label{eq:CFL:2DRMHD}
	0<    \frac{ \alpha \Delta t_n }{2 |{\mathcal I}_k|  } \sum_{j \in {\mathcal N}_k} |{\mathcal E}_{kj}|   < \beta_k,
	\end{equation}
	the solution  $\bar{\vec U}_k^{n+1}$ of the scheme \eqref{eq:2DRMHD:cellaverage} belongs to  ${\mathcal G}$.
\end{theorem}

\begin{proof}
	Let $\lambda_k :=  \frac{\alpha\Delta t_n }{2 |{\mathcal I}_k|  } \sum_{j \in {\mathcal N}_k} |{\mathcal E}_{kj}|  \in (0,\beta_k)$. If substituting  \eqref{eq:LFflux} into \eqref{eq:2DRMHD:cellaverage},
	then by technical arrangements one obtains the decomposition
	\begin{align}  \label{eq:2DRMHD:split:proof}
	\bar{\vec U}_{k}^{n+1}
	& = (1-2\beta_k) {\bf W}_k^n + 2 (\beta_k-\lambda_k) {\bf \Xi}_1 +  2 \lambda_k {\bf \Xi}_2,
	\end{align}
	with
	\begin{align*}
	{\bf \Xi}_1
	& := \frac{ 1 }{ \sum_{j \in {\mathcal N}_k} |{\mathcal E}_{kj}| }
	\sum_{j \in {\mathcal N}_k} |{\mathcal E}_{kj}| \bigg( \sum_{ \mu =1 }^Q \omega_\mu {\bf U}_k^n ({\bf x}_{kj}^{(\mu)}) \bigg),
	\\ \nonumber
	{\bf \Xi}_2
	& :=  \frac{ 1 }{ 2\sum_{j \in {\mathcal N}_k} |{\mathcal E}_{kj}| }
	\left\{
	\sum_{j \in {\mathcal N}_k} |{\mathcal E}_{kj}|
	\Bigg[
	\sum\limits_{\mu=1}^Q {{\omega _\mu}}
	\bigg( {\bf U}_k^n ( {\bf x}_{kj}^{(\mu)} ) -
	\alpha^{-1}	 \left\langle
	{\bm \xi}_{kj}, {\vec F} ( {\bf U}_k^n ( {\bf x}_{kj}^{(\mu)} ) )  \right\rangle
	\bigg)
	\Bigg] \right.
	\\
	& \qquad \left.+ \sum_{j \in {\mathcal N}_k} |{\mathcal E}_{kj}|
	\Bigg[
	\sum\limits_{\mu=1}^Q {{\omega _\mu}}
	\bigg( {\bf U}_j^n ( {\bf x}_{kj}^{(\mu)} ) -
	\alpha^{-1}\left\langle
	{\bm \xi}_{kj}, {\vec F} ( {\bf U}_j^n ( {\bf x}_{kj}^{(\mu)} ) )  \right\rangle
	\bigg)
	\Bigg] \right\}.
	\end{align*}
	Using the convexity of $\mathcal G$, one has ${\bf \Xi}_1 \in {\mathcal G}$ under the condition \eqref{eq:AdmiCon}. To show ${\bf \Xi}_2\in \overline{\mathcal G}$ by using Theorem \ref{theo:RMHD:LLFsplit2Dus},
	we investigate the corresponding DDF condition required in Theorem \ref{theo:RMHD:LLFsplit2Dus}.
	The required DDF condition is found to be
	\begin{align*}
	\mbox{\rm div} _{k}^{\mbox{\tiny \rm in}} {\vec B}^n
	+ \mbox{\rm div} _{k}^{\mbox{\tiny \rm out}} {\vec B}^n=0,
	\end{align*}
	which is equivalent to the condition \eqref{eq:DivB:cst}. Therefore,  under the conditions  \eqref{eq:DivB:cst}
	and \eqref{eq:AdmiCon}, Theorem \ref{theo:RMHD:LLFsplit2Dus} implies ${\bf \Xi}_2\in \overline{\mathcal G}$.
	Using \eqref{eq:2DRMHD:split:proof} and Corollary \ref{lam:new:convex} yields $ \bar{\vec U}_{k}^{n+1} \in {\mathcal G}$, and completes the proof.
\end{proof}

\begin{remark}
	It should be mentioned that Theorem \ref{thm:PCP:2DRMHD} also holds for $\beta_k=\frac12$, if
	the condition \eqref{eq:AdmiCon2} is replaced with
	\begin{equation}\label{eq:AdmiCon3}
	\bar{\bf U}_k^n - \frac{ 1 }{ \sum_{j \in {\mathcal N}_k} |{\mathcal E}_{kj}| }
	\sum_{j \in {\mathcal N}_k} |{\mathcal E}_{kj}| \bigg( \sum_{ \mu =1 }^Q \omega_\mu {\bf U}_k^n ({\bf x}_{kj}^{(\mu)}) \bigg) =0.
	\end{equation}
	In this case, the first term at the right-hand side of \eqref{eq:2DRMHD:split:proof}
	vanishes. This is similar in the following conclusion.
\end{remark}

\begin{remark}
	Theorem \ref{thm:PCP:2DRMHD} provides
	several sufficient conditions \eqref{eq:DivB:cst}--\eqref{eq:AdmiCon2} on the
	function $\vec U_{k}^n({\bf x})$ reconstructed
	in the finite volume method or evolved in the DG method
	in order to ensure that the  numerical schemes \eqref{eq:2DRMHD:cellaverage} is PCP.
	As we will discuss later, the conditions \eqref{eq:AdmiCon}--\eqref{eq:AdmiCon2} can be met by using the PCP
	limiter similar to that in Ref.\ \cite{WuTangM3AS2017}.
	The DDF condition \eqref{eq:DivB:cst} is milder than those in Ref.\ \cite{WuTangM3AS2017}, where two DDF
	conditions $\mbox{\rm div} _{k}^{\mbox{\tiny \rm in}} {\vec B}^n=0$ and $\mbox{\rm div} _{k}^{\mbox{\tiny \rm out}} {\vec B}^n=0 $ were needed.
	Using the divergence theorem, one can obtain $\mbox{\rm div} _{k}^{\mbox{\tiny \rm in}} {\vec B}^n=0$ if the polynomial vector ${\bf B}_{k}^n({\bf x})$ is locally divergence-free (cf. \cite{Li2005}). Such locally divergence-free property
	is not destroyed in the PCP limiting procedure
	since the PCP limiter
	modifies the vectors $\vec U_{k}^n({\bf x})$ with only a simple scaling.
	However, it is not easy to meet the DDF condition \eqref{eq:DivB:cst} or $\mbox{\rm div} _{k}^{\mbox{\tiny \rm out}} {\vec B}^n=0 $,  because they
	{depend} on the limiting values of  the magnetic field calculated from the neighboring cells ${\mathcal I}_j,~j\in {\mathcal N}_k$.
	If the polynomial vectors $\{{\bf B}_{k}^n({\bf x})\}$ are globally divergence-free,
	in other words, it is locally divergence-free in ${\mathcal I}_k$ with
	normal magnetic component continuous across the cell interface,
	then \eqref{eq:DivB:cst} is satisfied.
	But the PCP limiter with local scaling may destroy the globally divergence free property of ${\bf B}_{k}^n ({\bf x})$. Hence, it is nontrivial and still open to design a limiting procedure for the polynomial vector $\vec U_{k}^n({\bf x})$ for satisfying all the sufficient conditions \eqref{eq:DivB:cst}--\eqref{eq:AdmiCon2} at the same time.
	Fortunately, if the numerical magnetic field ${\bf B}_{k}^n({\bf x})$ is locally divergence-free and  converges to
	the exact solution within each cell, then refining the mesh may weaken the impact of violating \eqref{eq:DivB:cst} on the PCP property, see the numerical evidences in Ref.\ \cite{WuTangM3AS2017} on 2D uniform rectangular meshes and the interpretation in the following proposition.
\end{remark}

\begin{prop}
	Assume that ${\bf B}_{k}^n({\bf x})$ is locally divergence-free and approximates the exact solution ${\bf B}({\bf x},t_n)$ with at least first order accuracy within each cell ${\mathcal I}_{k}$,
	and the polynomial vectors $\{{\vec U}_{k}^n({\bf x})\}$ satisfy \eqref{eq:AdmiCon}
	and \eqref{eq:AdmiCon2} for a constant $\beta_k \in (0,\frac12)$.
	Then under the CFL type condition \eqref{eq:CFL:2DRMHD}, the solution  $\bar{\vec U}_k^{n+1}$ of the scheme \eqref{eq:2DRMHD:cellaverage} satisfies that $\bar D_k^{n+1}>0$, and for
	any ${\vec B}^*,\vec v^* \in \mathbb{R}^3$ with $|\vec v^*|<1$, it holds
	$$
	\bar{\vec U}_k^{n+1} \cdot \vec n^* + p_{m}^* >
	-\frac{ 2 \lambda_k {\vec v}^* \cdot  {\vec B}^*  }{\alpha{\sum\limits_{ j \in {\mathcal N}_k} | {\mathcal E}_{kj} | }} \mbox{\rm div} _{k}^{\mbox{\tiny \rm out}} {\vec B}^n
	= -{\mathcal O} (\Delta),
	$$
	where $\Delta$ is the biggest radius of the circumscribed $d$-spheres of the cells $\{{\mathcal I}_k\}$.
	This interprets that, as $\Delta$ approaches zero, $\bar{\bf U}_{k}^{n+1}$ may become more probably in ${\mathcal G}$.
\end{prop}

\begin{proof}
	Using the hypothesis
	and the continuity of exact normal magnetic field $\langle \bm \xi_{kj}, {\bf B}({\bf x},t_n)   \rangle$ across ${\mathcal E}_{kj}$ implies that $ \mbox{\rm div} _{k}^{\mbox{\tiny \rm in}} {\vec B}^n =0$, and
\begin{align*}
	& \left\langle \bm \xi_{kj}, {\bf B}_k^n({\bf x}_{kj}^{(\mu)}) \right\rangle = \left\langle \bm \xi_{kj}, {\bf B}({\bf x}_{kj}^{(\mu)},t_n) \right\rangle  + {\mathcal O} (\Delta),
	\\
	&
	\left\langle \bm \xi_{kj}, {\bf B}_j^n({\bf x}_{kj}^{(\mu)}) \right\rangle = \left\langle \bm \xi_{kj}, {\bf B}({\bf x}_{kj}^{(\mu)},t_n) \right\rangle  + {\mathcal O} (\Delta).
\end{align*}
	It follows that
	\begin{align*}
	\left|\mbox{\rm div} _{k}^{\mbox{\tiny \rm out}} {\vec B}^n -
	\mbox{\rm div} _{k}^{\mbox{\tiny \rm in}} {\vec B}^n \right|
	&\le  \sum\limits_{ j \in {\mathcal N}_k} \big| {\mathcal E}_{kj} \big|
	\sum\limits_{ \mu=1}^Q \omega_\mu
	\left| \left\langle \bm \xi_{kj}, {\bf B}_k^n ({\bf x}_{kj}^{(\mu)}) \right\rangle
	-
	\left\langle \bm \xi_{kj}, {\bf B}_j^n ({\bf x}_{kj}^{(\mu)}) \right\rangle
	\right|
	\\
	&= {\mathcal O} (\Delta) \sum_{ j \in {\mathcal N}_k}
	\big| {\mathcal E}_{kj} \big|.
	\end{align*}
	Hence $\mbox{\rm div} _{k} {\vec B}^n = \frac12 \mbox{\rm div} _{k}^{\mbox{\tiny \rm out}} {\vec B}^n= \frac12 \mbox{\rm div} _{k}^{\mbox{\tiny \rm in}} {\vec B}^n + {\mathcal O} (\Delta) \sum_{ j \in {\mathcal N}_k}
	\big| {\mathcal E}_{kj} \big| =  {\mathcal O} (\Delta) \sum_{ j \in {\mathcal N}_k}
	\big| {\mathcal E}_{kj} \big|$, so that ${\bf \Xi}_2$ may not belong to ${\mathcal G}$.
	However, ${\bf \Xi}_2$ is very close to ${\mathcal G}$ in the sense of that the first component of ${\bf \Xi}_2$ is positive, and for any ${\vec B}^*,\vec v^* \in \mathbb{R}^3$ with $|\vec v^*|<1$, it holds
	\begin{align*}
	{\bf \Xi}_2 \cdot \vec n^* + p_{m}^* &\ge
	-\frac{  {\vec v}^* \cdot  {\vec B}^*  }{\alpha{\sum\limits_{ j \in {\mathcal N}_k} | {\mathcal E}_{kj} | }}\sum\limits_{ j \in {\mathcal N}_k}    \left\langle \bm \xi_{kj}, {\bf B}_j^n({\bf x}_{kj}^{(\mu)}) \right\rangle   \big| {\mathcal E}_{kj} \big|
	\\
	&
	= -\frac{  {\vec v}^* \cdot  {\vec B}^*  }{\alpha{\sum\limits_{ j \in {\mathcal N}_k} | {\mathcal E}_{kj} | }} \mbox{\rm div} _{k}^{\mbox{\tiny \rm out}} {\vec B}^n  = -{\mathcal O} (\Delta),
	\end{align*}
	{where the derivation of the inequality is  similar to that of Theorem \ref{theo:RMHD:LLFsplit2Dus}.}
	Because ${\bf \Xi}_1 \in {\mathcal G}$, one concludes from
	\eqref{eq:2DRMHD:split:proof} that $\bar D_{k}^{n+1}>0$ and
	\begin{align*}
	& \bar{\vec U}_{k}^{n+1} \cdot \vec n^* + p_{m}^*
	= (1-2\beta_k) \big({\bf W}_k^n\cdot \vec n^* + p_{m}^* \big)
	\\
	& \quad
	+ 2 (\beta_k-\lambda_k) \big({\bf \Xi}_1 \cdot \vec n^* + p_{m}^* \big)  +  2 \lambda_k \big( {\bf \Xi}_2 \cdot \vec n^* + p_{m}^* \big)
	\\
	&> 2 \lambda_k \big( {\bf \Xi}_2 \cdot \vec n^* + p_{m}^* \big)
	\ge -\frac{ 2 \lambda_k {\vec v}^* \cdot  {\vec B}^*  }{\alpha{\sum\limits_{ j \in {\mathcal N}_k} | {\mathcal E}_{kj} | }} \mbox{\rm div} _{k}^{\mbox{\tiny \rm out}} {\vec B}^n  = -{\mathcal O} (\Delta).
	\end{align*}
	The proof is completed.
\end{proof}

Let us further understand the result in Theorem \ref{thm:PCP:2DRMHD} on two special meshes in two dimension ($d=2$), and show that the conditions \eqref{eq:AdmiCon}--\eqref{eq:AdmiCon2} can be met by a PCP limiter.
In the 2D case, we use the $ Q$-point Gauss quadrature rule for the line integrations in \eqref{eq:quadfulx}.
For the accuracy requirement,   $Q$ should satisfy
$Q \ge K+1$ for a $\mathbb{P}^K$-based DG method \cite{Cockburn0},
or $Q \ge (K+1)/2$ for a $(K+1)$-th order accurate finite volume scheme.

\noindent
{\bf Example 1.} Assume $\mathcal M$ is a uniform rectangular mesh with cells
$\{{\mathcal I}_{ij}=[x_{i-\frac12},x_{i+\frac12}] \times [y_{j-\frac12},y_{j+\frac12}]\}$
and spatial step-sizes $\Delta x$ and $\Delta y$ in $x$- and $y$-directions respectively.
Let $\bar {\bf U}_{ij}^n $ and $\bar {\bf U}_{ij}^n({\bf x})$  denote
the approximate cell-averaged value and polynomial vector over the cell ${\mathcal I}_{ij}$, respectively, and
$ \mathbb{S}_i^x = \{  {x}_i^{(\mu)} \}_{\mu=1}^{ Q}$ and $ \mathbb{S}_j^{y} =  \{  {y}_j^{(\mu)} \}_{\mu=1}^{ Q}$ denote the $ Q$-point Gauss quadrature nodes in the intervals $[ { x}_{i-\frac12}, { x}_{i+\frac12} ]$ and $[ { y}_{j-\frac12}, { y}_{j+\frac12} ]$, respectively.
Then the discrete divergence defined in \eqref{eq:DisDivB:cst} becomes
\begin{align*}
\mbox{\rm div} _{ij} {\bf B}^n & = \Delta x \Delta y
\Bigg( \frac1{\Delta x} {\sum\limits_{\mu=1}^{ Q} \omega_\mu \left(   \overline{(B_1)}_{i+\frac{1}{2},j}^{n,(\mu)}
	- \overline{(B_1)}_{i-\frac{1}{2},j}^{n,(\mu)}  \right)}
\\
& \quad + \frac1{\Delta y} {\sum \limits_{\mu=1}^{ Q} \omega_\mu \left(  \overline{ ( B_2)}_{i,j+\frac{1}{2}}^{n,(\mu)}
	-  \overline{ ( B_2)}_{i,j-\frac{1}{2}}^{n,(\mu)}    \right)} \Bigg), 
\end{align*}
where
\begin{align*}
&\overline{(B_1)}_{i+\frac{1}{2},j}^{n,(\mu)} := \frac12 \Big( (B_1)_{ij}^n ({ x}_{i+\frac12},{ y}_j^{(\mu)}) + (B_1)_{i+1,j}^n ({ x}_{i+\frac12},{ y}_j^{(\mu)}) \Big),
\\
&\overline{ ( B_2) }_{i,j+\frac{1}{2}}^{n,(\mu)}  :=
\frac12 \Big( (B_2)_{ij}^n ({ x}_i^{(\mu)},{ y}_{j+\frac12}) + (B_2)_{i,j+1}^n ({ x}_i^{(\mu)},{ y}_{j+\frac12}) \Big).
\end{align*}
Let $\hat{\mathbb{S}}_i^x=\{ \hat { x}_i^{(\nu) }\}_{\nu=1} ^ { L}$ and $\hat{\mathbb{S}}_j^y=\{ \hat { y}_j^{(\nu)} \}_{\nu=1} ^{L}$ be the $ L$-point Gauss-Lobatto quadrature nodes in the intervals
$[{ x}_{i-\frac{1}{2}},{ x}_{i+\frac{1}{2}}]$ and $[{ y}_{j-\frac{1}{2}},{ y}_{j+\frac{1}{2}} ]$ respectively, and
$ \{\hat \omega_\nu\}_{\nu=1} ^ { L}$ be  associated weights satisfying $\sum_{\nu=1}^{ L} \hat\omega_\nu = 1$, where  ${ L}\ge \frac{{\tt K}+3}2$ such that the
associated quadrature has algebraic precision of at least degree $K$.

As a corollary of Theorem \ref{thm:PCP:2DRMHD}, the following
conclusion improves the result in Ref.\ \cite{WuTangM3AS2017} where two DDF
conditions $\mbox{\rm div} _{ij}^{\mbox{\tiny \rm in}} {\vec B}^n=0$ and $\mbox{\rm div} _{ij}^{\mbox{\tiny \rm out}} {\vec B}^n=0 $ were needed.

\begin{corollary}
	Let the mesh $\mathcal M$ consist of uniform rectangular cells ${\mathcal I}_{ij}$ and
	assume that the polynomial vectors
	$\{{\vec U}_{ij}^n({\bf x})\}$ satisfy
	\begin{gather}\label{eq:DivB:cstM1}
	\mbox{\rm div} _{ij}  {\bf B}^n=0, \quad \forall i,j,
	\\
	\label{eq:AdmiConM1}
	{\vec U}_{ij}^n ({\bf x}) \in {\mathcal G}, \quad \forall {\bf x}\in \mathbb{S}_{ij},~\forall i,j,
	\end{gather}
	where the set $
	\mathbb{S}_{ij}:= \big(\hat{\mathbb{S}}_i^x \otimes {\mathbb{S}}_j^y\big) \cup
	\big({\mathbb{S}}_i^x \otimes \hat{\mathbb{S}}_j^y\big)$ consists of several important quadrature nodes in the cell ${\mathcal I}_{ij}$.
	Then under the CFL type condition
	\begin{equation}\label{eq:CFL:2DRMHDM1}
	0< \alpha \Delta t_n \bigg( \frac{1}{\Delta x} + \frac{1}{\Delta y} \bigg)  < \hat \omega_1,
	\end{equation}
	the solution  $\bar{\vec U}_{ij}^{n+1}$ of the scheme \eqref{eq:2DRMHD:cellaverage} belongs to  ${\mathcal G}$.
\end{corollary}

\begin{proof}
	It only needs to verify that \eqref{eq:AdmiConM1} can ensure the conditions \eqref{eq:AdmiCon} and \eqref{eq:AdmiCon2} in Theorem \ref{thm:PCP:2DRMHD}.
	
	Note that $\hat x_i^{(1)}=x_{i-\frac12}$ and $\hat x_i^{(L)}=x_{i+\frac12}$,
	which implies that
	$\{x_{i-\frac12},x_{i+\frac12}\} \otimes {\mathbb{S}}_j^y \subseteq  \mathbb{S}_{ij}.$
	Similarly, one has $
	{\mathbb{S}}_i^x \otimes \{y_{j-\frac12},y_{j+\frac12}\} \subseteq  \mathbb{S}_{ij}.
	$
	Thus the condition \eqref{eq:AdmiConM1} implies \eqref{eq:AdmiCon}.
	
	The following shows that the condition \eqref{eq:AdmiCon2} can be met for $\beta_k=\hat \omega_1$.
	Using the exactness of the Gauss-Lobatto quadrature rule with $ L$ nodes and the Gauss quadrature rule with $ Q$ nodes for the polynomials of degree $ K$,
	one can decompose (cf. \cite{WuTangM3AS2017} for details) the cell average value into
	\begin{equation*} 
	\bar{\bf U}_{ij}^n = \sum_{\nu=2}^{L-1} \hat \omega_\nu   {\bf \Pi}_{\nu} + 2 \hat \omega_1 {\bf \Pi}_1,
	\end{equation*}
	where $\hat \omega_1 = \hat \omega_{ L}$ has been used, and
	\begin{align*}
	&{\bf \Pi}_{\nu} :=  \sum \limits_{\mu = 1}^{Q}  \omega_\mu \bigg(
	\frac{\Delta y}{\Delta x+ \Delta y}    {\bf U}_{ij}^n\big(\hat { x}_i^{(\nu)},{y}_j^{(\mu)}\big) + \frac{\Delta x}{\Delta x+ \Delta y}   {\bf U}_{ij}^n\big( { x}_i^{(\mu)},\hat { y}_j^{(\nu)} \big) \bigg),~2\le \nu < L,
	\\
	&
	{\bf \Pi}_1 := \frac{1}{2(\Delta x + \Delta y)} \bigg(
	\Delta y \sum \limits_{\mu = 1}^{\tt Q}  \omega_\mu {\bf U}_{ij}^n\big({ x}_{i-\frac12},{y}_j^{(\mu)}\big)
	+ \Delta y \sum \limits_{\mu = 1}^{\tt Q}  \omega_\mu {\bf U}_{ij}^n\big({ x}_{i+\frac12},{y}_j^{(\mu)}\big)
	\\
	& \qquad \quad +\Delta x \sum \limits_{\mu = 1}^{\tt Q}  \omega_\mu {\bf U}_{ij}^n\big( { x}_i^{(\mu)}, { y}_{j-\frac12} \big)
	+  \Delta x \sum \limits_{\mu = 1}^{\tt Q}  \omega_\mu
	{\bf U}_{ij}^n\big( { x}_i^{(\mu)}, { y}_{j+\frac12} \big)
	\bigg).
	\end{align*}
	Thanks to the convexity of $\mathcal G$, one has ${\bf \Pi}_{\nu} \in {\mathcal G}$ by \eqref{eq:AdmiConM1}.
	Hence there exists a constant $\beta_k = \hat \omega_1$ such that
	$$
	{\bf W}_{ij}^n:=\frac{1}{1-2\beta_k} \left( \bar{\bf U}_{ij}^n - 2 \beta_k {\bf \Pi}_1 \right)
	=
	\sum_{\nu=2}^{L-1} \frac{\hat \omega_\nu}{1-2\hat \omega_1}   {\bf \Pi}_{\nu},
	$$
	which belongs to $\mathcal G$. Hence \eqref{eq:AdmiCon2} is satisfied for $\beta_k=\hat \omega_1$.
	Using Theorem \ref{thm:PCP:2DRMHD} completes the proof.
\end{proof}

\noindent
{\bf Example 2.} Assume that $\mathcal M$ is a 2D triangular mesh.
Following the approach in Ref.\ \cite{zhang2012}, one
can decompose the cell average $\bar{\bf U}_k^n$
into a convex combination of point values of the polynomial  ${\bf U}_k^n({\bf x})$
by a 2D quadrature satisfying:
\begin{itemize}
	\item The quadrature rule is with positive weights and exact for integration of ${\bf U}_k^n({\bf x})$ in the cell ${\mathcal I}_k$.
	\item The set of the 2D quadrature nodes, denoted by $\mathbb{S}_k$, should include
	all the Gauss quadrature nodes ${\bigcup} _{j\in{\mathcal N}_{k}} \{ {\bf x}_{kj}^{(\mu)},\mu=1,\cdots,Q \}$ on the edges of ${\mathcal I}_k$.
\end{itemize}
A qualified 2D quadrature rule was constructed in Ref.\ \cite{zhang2012}
and summarized below.
Let $\{\zeta_\mu \}_{\mu=1}^Q$ denote the Gauss quadrature nodes on $[-\frac12,\frac12]$
and $\{\hat \zeta_\nu \}_{\nu=1}^L$ be the Gauss-Lobatto quadrature nodes on $[-\frac12,\frac12]$, and $\{\hat \omega_\nu\}_{\nu=1} ^ { L}$ be  associated weights satisfying $\sum_{\nu=1}^{ L} \hat\omega_\nu = 1$, where  ${ L}\ge \frac{{\tt K}+3}2$ such that the
associated quadrature has algebraic precision of at least degree $K$.
Then the set of local barycentric coordinates of the 2D quadrature nodes in $\mathbb{S}_k$ is
given by
\begin{align*}
& \Bigg\{
\bigg( \frac12 + \zeta_\mu, (\frac12 + \hat\zeta_\nu)(\frac12 - \zeta_\mu),
(\frac12-\hat \zeta_{\nu})(\frac12 - \zeta_{\mu})   \bigg), \\
& \quad
\bigg( (\frac12 - \hat\zeta_\nu)(\frac12 - \zeta_\mu), \frac12 + \zeta_\mu,
(\frac12+\hat \zeta_{\nu})(\frac12 - \zeta_{\mu})   \bigg),
\\
&
\quad
\bigg( (\frac12 + \hat\zeta_\nu)(\frac12 - \zeta_\mu),
(\frac12-\hat \zeta_{\nu})(\frac12 - \zeta_{\mu}) , \frac12 + \zeta_\mu  \bigg),
1\le \mu \le Q, 1\le \nu \le L
\Bigg\}.
\end{align*}
The cell average value $\bar{\bf U}_k^n$ can be written as
\begin{equation}\label{eq:tridecop}
\begin{split}
\bar{\bf U}_k^n &= \frac{1}{|{\mathcal I}_k|} \int_{{\mathcal I}_{k}}
{\bf U}_k^n({\bf x}) {d {\bf x}}
= \sum_{ {\bf x} \in \mathbb{S}_k }  \varpi_{\bf x}  {\bf U}_k^n({\bf x})
\\
&
= \frac{2}{3} \hat \omega_1 \sum_{j \in {\mathcal N}_k}  \bigg( \sum_{ \mu =1 }^Q \omega_\mu {\bf U}_k^n ({\bf x}_{kj}^{(\mu)}) \bigg)
+ \sum_{ {\bf x} \in \mathbb{S}_k^{\mbox{\tiny \rm int}} }  \varpi_{\bf x}  {\bf U}_k^n({\bf x}),
\end{split}
\end{equation}
where  $\varpi_{\bf x}$ is the quadrature weight for
the node ${\bf x} \in \mathbb{S}_k $,  and $\mathbb{S}_k^{ \mbox{\tiny \rm int} } $
is the set of the points in
$\mathbb{S}_k$ that lie in
the interior of ${\mathcal I}_k$. The specific expressions of $\varpi_{\bf x}$
for ${\bf x} \in \mathbb{S}_k^{ \mbox{\tiny \rm int} } $ are omitted here, because we only use $\varpi_{\bf x}>0$ and $\sum_{ {\bf x} \in \mathbb{S}_k^{\mbox{\tiny \rm int}} }  \varpi_{\bf x}=1-2 \hat \omega_1$.

As a corollary of Theorem \ref{thm:PCP:2DRMHD}, the following conclusion holds.

\begin{corollary}
	Let the mesh $\mathcal M$ consist of triangular cells ${\mathcal I}_{k}$.
	If the polynomial vectors
	$\{{\vec U}_{k}^n({\bf x})\}$ satisfy the DDF condition \eqref{eq:DivB:cst},
	and
	\begin{gather}
	\label{eq:AdmiConM2}
	{\vec U}_{k}^n ({\bf x}) \in {\mathcal G}, \quad \forall {\bf x}\in \mathbb{S}_{k},~\forall k,
	\end{gather}
	then under the CFL type condition
	\begin{equation}\label{eq:CFL:2DRMHDM2}
	0<    \frac{ \alpha \Delta t_n }{2 |{\mathcal I}_k|  } \sum_{j \in {\mathcal N}_k} |{\mathcal E}_{kj}|  < \beta_k,
	\end{equation}
	the solution  $\bar{\vec U}_{k}^{n+1}$ of the scheme \eqref{eq:2DRMHD:cellaverage} belongs to  ${\mathcal G}$, where
	\begin{equation}\label{eq:def:betak}
	\beta_k:=\frac{\hat \omega_1}{3} \times
	\frac{ \sum_{j \in {\mathcal N}_k} |{\mathcal E}_{kj}| }{ \max_{j \in {\mathcal N}_k} |{\mathcal E}_{kj}| } \in \left( \frac{2\hat \omega_1}{3}, \hat \omega_1  \right].
	\end{equation}
\end{corollary}

\begin{proof}
	It only needs to show that the conditions \eqref{eq:AdmiCon} and \eqref{eq:AdmiCon2} in Theorem \ref{thm:PCP:2DRMHD} are ensured by \eqref{eq:AdmiConM2} for the constant $\beta_k$.
	Because ${\bigcup}_{j\in{\mathcal N}_{k}} \{ {\bf x}_{kj}^{(\mu)},\mu=1,\cdots,Q \} \subseteq \mathbb{S}_{k}$, the condition \eqref{eq:AdmiConM2} implies \eqref{eq:AdmiCon}.
	The following verifies the condition \eqref{eq:AdmiCon2} for $\beta_k$ in \eqref{eq:def:betak}. Denote ${\mathcal E}_k^\infty:= \max_{j \in {\mathcal N}_k} |{\mathcal E}_{kj}|$, then using \eqref{eq:tridecop} gives
	\begin{align*}
	(1-2\beta_k){\bf W}_k^n & =\bar{\bf U}_k^n - \frac{ 2 \hat \omega_1 }{ 3 {\mathcal E}_k^\infty }
	\sum_{j \in {\mathcal N}_k} |{\mathcal E}_{kj}| \bigg( \sum_{ \mu =1 }^Q \omega_\mu {\bf U}_k^n ({\bf x}_{kj}^{(\mu)}) \bigg)
	\\
	& = \frac{2}{3} \hat \omega_1 \sum_{j \in {\mathcal N}_k}
	\bigg(1-\frac{ |{\mathcal E}_{kj}| }{{\mathcal E}_k^\infty} \bigg)
	\bigg( \sum_{ \mu =1 }^Q \omega_\mu {\bf U}_k^n ({\bf x}_{kj}^{(\mu)}) \bigg)
	+ \sum_{ {\bf x} \in \mathbb{S}_k^{\mbox{\tiny \rm int}} }  \varpi_{\bf x}  {\bf U}_k^n({\bf x}),
	\end{align*}
	which implies that ${\bf W}_k^n$ is a convex combination of several admissible states.
	Hence ${\bf W}_k^n \in {\mathcal G}$ by the convexity of $\mathcal G$. The proof is completed.
\end{proof}

\begin{remark}
	It can be similarly shown that the bound-preserving high-order methods in Ref.\ \cite{zhang2012} allow a CFL condition like \eqref{eq:CFL:2DRMHDM2}, which is milder
	than that used in Ref.\ \cite{zhang2012} (corresponding to $\beta_k = \hat{ \omega}_1/3$ here), see also the derivation in Ref.\ \cite{Zhang2017}.
\end{remark}

For an arbitrary 2D mesh consisting of non-self-intersecting polygons, one can
subdivide ${\mathcal I}_k$ into several non-overlapping sub-domains such
that those subdomains which share edges with $\{{\mathcal E}_{kj}\}_{j \in {\mathcal N}_k}$
are all triangles.
Then a 2D quadrature rule can be constructed to decompose the cell average $\bar{\bf U}_k^n$
into a convex combination of point values of the polynomial  ${\bf U}_k^n({\bf x})$
at some quadrature nodes ${\mathbb{S}}_k$,
so that the conditions \eqref{eq:AdmiCon} and \eqref{eq:AdmiCon2}
in Theorem \ref{thm:PCP:2DRMHD} can be ensured under an admissibility condition like \eqref{eq:AdmiConM2}. The details are similar to those in Ref.\ \cite{DuShu2017}
and omitted here.


The condition \eqref{eq:AdmiConM2} or \eqref{eq:AdmiConM1} can be enforced by a PCP limiting procedure, in which the polynomial vector
${\vec U}^n_k({\bf x})$ is limited as $\widetilde{\vec U}_k({\bf x})$ such that $\widetilde{\vec U}_k ( {\bf x} ) \in {\mathcal G} $ for all ${\bf x} \in \mathbb{S}_k$.
To avoid the effect of the rounding error, we introduce a sufficiently small positive number  $\epsilon$ such that $ \bar {\vec U}_k^n \in {\mathcal G}_\epsilon$, where
\begin{align} \label{set-G-epsilon}
{\mathcal G}_\epsilon = \left\{   \vec U=(D,\vec m,\vec B,E)^{\top} \Big|  D\ge\epsilon,~q(\vec U)\ge\epsilon,~{\Psi}_\epsilon(\vec U) \ge 0\right\},
\end{align}
with
$$
{\Psi}_\epsilon(\vec U) : = {\Psi}(\vec U_\epsilon),\quad \vec U_\epsilon : = \big(D,\vec m,\vec B,E-\epsilon\big)^{\top}.
$$
%
%
%
Then the PCP limiting procedure is divided into the following three steps.

\noindent
{\bf Step (i)}: Enforce the positivity of $D(\vec U)$. Let $D_{\min} = \min \limits_{ {\bf x} \in {\mathbb S}_k}^{} D_k^n ( {\bf x} )$.
If $D_{\min} < \epsilon$,
then  $D_k^n ( {\bf x} )$ is limited as
$$
\hat D_k({\bf x}) = \theta_1 \big( D_k^n({\bf x}) - \bar D_k^n \big) + \bar D_k^n,
$$
where $\theta_1 = (\bar D_k^n - \epsilon)/ ( \bar D_k^n - D_{\min} ) $. Otherwise, take $\hat D_k({\bf x}) =  D_k^n({\bf x})$.
Denote
$$\hat {\vec U}_k({\bf x}) :=  \left( \hat D_k({\bf x}), \vec m_k^n({\bf x}),\vec B_k^n({\bf x}), E_k^n({\bf x}) \right)^{\top}.$$

\noindent
{\bf Step (ii)}: Enforce the positivity of $q(\vec U)$. Let
$q_{\min} = \min \limits_{{\bf x} \in {\mathbb S}_k}^{} q(\hat {\vec U}_k ( {\bf x} ))$. If $q_{\min} < \epsilon$, then limit $\hat {\vec U}_k ( {\bf x} )$ as
\begin{align*}
\check {\vec U}_k({\bf x}) &= \Big(  \theta_2 \big( \hat {D}_k ({\bf x}) - \bar {D}_k^n \big) + \bar {D}_k^n,
\theta_2 \big(  {\vec m}_k^n ({\bf x}) - \bar {\vec m}_k^n \big) + \bar {\vec m}_k^n,
\\
& \qquad
{\vec B}_k^n ({\bf x}) ,
\theta_2 \big(  {E}_k^n ({\bf x}) - \bar {E}_k^n \big) + \bar {E}_k^n \Big)^\top,
\end{align*}
where $\theta_2 = (q(\bar {\vec U}_k^n) - \epsilon)/ ( q(\bar {\vec U}_k^n) - q_{\min} ) $. Otherwise, set $\check {\vec U}_k({\bf x}) =  \hat {\vec U}_k({\bf x})$.

\noindent
{\bf Step (iii)}: Enforce the positivity of ${\Psi}(\vec U)$.
For each ${\bf x} \in {\mathbb S}_k$, if ${\Psi}_\epsilon ( \check {\vec U}_k({\bf x}) ) < 0$, then define $\tilde \theta({\bf x})$ by solving the nonlinear equation
\begin{equation}\label{eq:limiterEq}
{\Psi}_\epsilon \Big( (1- \tilde  \theta) \bar{\vec U}_k^n + \tilde  \theta \check {\vec U}_k({\bf x}) \Big) =0, \quad  \tilde  \theta \in [0,1).
\end{equation}
Otherwise,
set $\tilde  \theta({\bf x})=1$. Let $\theta_3 = \min \limits_{{\bf x}\in {\mathbb S}_k} \{\tilde  \theta({\bf x})\}$  and
\begin{equation}\label{eq:PCPpolynomial}
\widetilde {\vec U}_k({\bf x}) = \theta_3 \big( \check {\vec U}_k ({\bf x}) - \bar {\vec U}_k^n \big) + \bar {\vec U}_k^n.
\end{equation}

\begin{remark}
	The above analyses and proofs  can be directly applied to
	any other system if the admissible state set of the system
	has the generalized LxF splitting property
	and convexity. {Therefore, the  analyses can also be applied} to
	the numerical schemes for the RHDs {on general meshes}. For the RHDs which correspond to
	RMHDs with zero magnetic field, the above theorems and corollaries
	hold without the DDF condition.
\end{remark}

\section{Conclusions}\label{sec:con}

The ideal gas EOS with a constant adiabatic index is a poor approximation for most relativistic astrophysical flows, although it is commonly used in {the relativistic magnetohydrodynamics (RMHD)}.
The paper   extended the physical-constraints-preserving (PCP) analysis in Ref.\ \cite{WuTangM3AS2017}
to the numerical schemes for the multi-dimensional RMHDs with a general EOS  on a general mesh
(with  non-self-intersecting polytopes).
It was built on several important properties of the admissible state set,
whose derivations for a general EOS was nontrivial and partly involved analytically handling a EOS function without a specific expression.
Based on those properties, we   provided rigorous PCP analyses of finite volume or discontinuous Galerkin schemes with the Lax-Friedrichs type flux.
The results  showed that in the general case,
there also existed a divergence-free condition in discrete sense that was critical
for designing the PCP schemes.
Moreover, the discrete divergence-free condition was  proposed {on general meshes} and milder than that in Ref.\ \cite{WuTangM3AS2017}.
It {is} worth emphasizing that the present analyses  can be directly applied to the relativistic hydrodynamic equations (without  magnetic field).

Our future work will include numerical experiments to further demonstrate
the theoretical analyses, and
the exploration of numerical techniques which can meet the discrete divergence-free condition
and also work in conjunction with the PCP limiter.
%

\end{document}